\documentclass[11pt]{amsart}

\usepackage{algorithm}
\usepackage[utf8]{inputenc}
\usepackage[T1]{fontenc}
\usepackage[english]{babel}
\usepackage{amsmath,amssymb,amsfonts,amsthm}
\usepackage{amsmath,amsthm,amsfonts,amssymb,eucal, hyperref}
\usepackage[pdftex]{graphicx}
\usepackage{xcolor}
\usepackage{pseudocode}
\usepackage{mathtools}
\usepackage[normalem]{ulem}
\mathtoolsset{showonlyrefs}

\newtheorem{thm}{Theorem}[section]
\newtheorem{lem}[thm]{Lemma}
\newtheorem{prop}[thm]{Proposition}
\newtheorem{cor}[thm]{Corollary}
\newtheorem{rem}[thm]{Remark}
\newtheorem*{open}{Open problem}
\theoremstyle{definition}
\newtheorem{defi}[thm]{Definition}

\newcommand{\del}{\partial}

\newcommand{\N}{\mathbb{N}}

\newcommand{\R}{\mathbb{R}}

\newcommand{\CM}{\mathcal{M}}

\newcommand{\CR}{\mathcal{R}}

\newcommand{\abs}[1]{\left\lvert #1 \right\rvert}
\newcommand{\set}[1]{\left\{ #1 \right\}}

\newcommand{\ds}{\displaystyle}
\newcommand{\dsum}{\ds\sum}
\newcommand{\bo}\boldsymbol{}
\newcommand{\bigo}[1]{\mathrm{O}\left( #1 \right)}
\newcommand{\eqskip}{ \vspace*{2mm}\\ }
\newcommand{\smallo}[1]{o\left( #1 \right)}
\newcommand\soutb{\bgroup\markoverwith{\textcolor{blue}{\rule[.5ex]{2pt}{1pt}}}\ULon}

\newcommand{\fr}[2]{\frac{\ds #1}{\ds #2}}
\newcommand{\wv}[1]{\omega_{ #1}}
\newcommand{\so}[1]{{\rm o}\left(#1\right)}

\renewcommand{\hat}{\widehat}
\renewcommand{\tilde}{\widetilde}
\newcommand{\eps}{\varepsilon}
\renewcommand{\phi}{\varphi}

\renewcommand{\bar}[1]{\overline{#1}}

\begin{document}

\title[Optimal copies and P\'olya's conjecture]
{Optimal unions of scaled copies of domains and P\'olya's conjecture}

\author[Pedro Freitas]{Pedro Freitas}
\author[Jean Lagac\'e]{Jean Lagac\'e}
\author[Jordan Payette]{Jordan Payette}

\address{Departamento de Matem\'{a}tica, Instituto Superior T\'{e}cnico, Universidade de Lisboa, Av. Rovisco Pais, 1049-001 Lisboa, Portugal
\& Grupo de F\'{\i}sica Matem\'{a}tica, Faculdade de Ci\^{e}ncias, Universidade de Lisboa, Campo Grande, Edif\'{\i}cio C6,
1749-016 Lisboa, Portugal}
\email{psfreitas@fc.ul.pt}

\address{Department of Mathematics, University College London, Gower Street,
WC1E 6BT, London, United Kingdom}
\email{j.lagace@ucl.ac.uk}

\address{D\'epartement de math\'ematiques et de statistique \\ Universit\'e de Montr\'eal \\ C. P. 6128,
Succ. Centre-ville \\ Montr\'eal, QC\\ H3C 3J7 \\ Canada}

\email{payettej@dms.umontreal.ca}

 \begin{abstract}
 Given a bounded Euclidean domain $\Omega$, we consider the sequence of optimisers of the $k^{\rm th}$ Laplacian eigenvalue
 within the family consisting of all possible disjoint unions of scaled copies
 of $\Omega$ with fixed total volume. We show that this sequence
 encodes information yielding conditions for $\Omega$ to satisfy P\'{o}lya's
 conjecture with either Dirichlet or Neumann boundary conditions. This is an extension of a result by Colbois and
 El Soufi which applies only to the case
 where the family of domains consists of {\it all} bounded domains. Furthermore, we fully classify the different possible behaviours for
 such sequences, depending on whether P\'{o}lya's conjecture holds for a given specific domain or not.
 This approach allows us to recover a stronger version of P\'{o}lya's original results for tiling
 domains satisfying some dynamical billiard conditions, and a strenghtening of Urakawa's bound in terms of packing density.
 \end{abstract}

\maketitle

\section{Introduction and main results}

\subsection{P\'olya's conjecture for Laplace eigenvalues}

For $d \ge 2$ let $\Omega \subset \R^d$ be a bounded open set with Lebesgue measure $|\Omega|$. 
We consider the Dirichlet eigenvalue problem
\begin{equation}
 \begin{cases}
  \Delta u + \lambda u = 0  & \text{in } \Omega\\
  u \equiv 0 &\text{on } \del \Omega.
 \end{cases}
\end{equation}
It is well known that the eigenvalues of the above problem are discrete and form a sequence
\begin{equation}
 0 < \lambda_1(\Omega) \le \lambda_2(\Omega) \le \lambda_3(\Omega) \dotso \nearrow \infty
\end{equation}
accumulating only at infinity.
Moreover, if the boundary $\del \Omega$ is Lipschitz, the Neumann problem
\begin{equation}
 \begin{cases}
  \Delta u + \mu u = 0  & \text{in } \Omega\\
  \del_\nu u \equiv 0 &\text{on } \del \Omega,
 \end{cases}
\end{equation}
where $\nu$ denotes the outer unit normal vector of $\Omega$, also has discrete spectrum and forms an increasing sequence
\begin{equation}
  0 = \mu_0(\Omega) \le \mu_1(\Omega) \le \dotso \nearrow \infty.
\end{equation}
Note that we choose the convention to start numbering Neumann eigenvalues
with $0$ instead of with $1$, which allows for a
cleaner statement of our theorems.
Both the Dirichlet and Neumann eigenvalues satisfy so-called \emph{Weyl asymptotics}
  \begin{equation}
    \lambda_k = \mu_k + \bigo{k^{1/d}} = \frac{4\pi^2}{(\wv d \abs \Omega)^{2/d}}k^{2/d} + \bigo{k^{1/d}},
  \end{equation}
  where $\wv d$ denotes the volume of the unit ball in $\R^d$.
If $\Omega$ satisfies some dynamical conditions, namely that the measure of
periodic trajectories in the billiard flow is zero, the eigenvalues also satisy
two-term Weyl asymptotics~\cite{sava}
\begin{equation}
 \label{twotermweyldir}
 \lambda_{k} = \fr{4\pi^2 }{\left(\wv{d} |\Omega|\right)^{2/d}}k^{2/d} +
 \fr{2 \pi^2}{d}
 \fr{\wv{d-1}\left|\partial\Omega\right|}{\left(\wv{d}\left|\Omega\right|\right)^{\frac{d+1}{d}}}k^{1/d} + \so{k^{1/d}}
\end{equation}
and
\begin{equation}
 \label{twotermweylneu}
 \mu_{k} = \fr{4\pi^2 }{\left(\wv{d} |\Omega|\right)^{2/d}}k^{2/d} -
 \fr{2 \pi^2}{d}
 \fr{\wv{d-1}\left|\partial\Omega\right|}{\left(\wv{d}\left|\Omega\right|\right)^{\frac{d+1}{d}}}k^{1/d} + \so{k^{1/d}}.
\end{equation}
From these asymptotic formulae it is clear that given a domain
$\Omega$ for which \eqref{twotermweyldir} and \eqref{twotermweylneu} hold there exists $k^{*}=k^{*}(\Omega)$ such that for all $k \ge k^*$,
\begin{equation}
  \label{eq:strict}
  \mu_k(\Omega) < \frac{4 \pi^2}{(\omega_d|\Omega|)^{2/d}} k^{2/d} < \lambda_k(\Omega).
\end{equation}
Furthermore, the Rayleigh--Faber--Krahn~\cite{fab,krahn1} and the Hong--Krahn--Szego~\cite{krahn2} inequalities imply that the right-hand side
inequality holds for $\lambda_{1}$ and $\lambda_{2}$, while the Szeg\H{o}--Weinberger~\cite{wein} and the Bucur--Henrot~\cite{buhe} inequalities ensure
the inequality on the left-hand side for $\mu_1$ and $\mu_{2}$. In this paper, we investigate a conjecture of P\'olya.
\begin{open}[P\'olya's conjecture]
  For all $\Omega\subset \R^d$ and all $k \in \N$,
 \begin{equation}\label{PolyaConj}
  \mu_k(\Omega) \le  \frac{4 \pi^2}{(\omega_d|\Omega|)^{2/d}} k^{2/d} \le \lambda_k(\Omega).
 \end{equation}
\end{open}\vspace{6pt}
In $1961$ P\'{o}lya proved that the above inequalities do hold for all domains which tile the plane, and conjectured
that this would be true for general domains~\cite{poly} --- see~\cite{kell} for
the proof for general tiling domains with Neumann boundary conditions. 
P\'{o}lya's result was later extended to tiling domains in higher dimensions by Urakawa, who also obtained lower bounds
for all Dirichlet eigenvalues of a domain based on its lattice packing density~\cite{urak}.

For general domains, the best results so far remain those by Berezin~\cite{B} and  Li and Yau~\cite{LY} in the Dirichlet case,
while for Neumann eigenvalues the corresponding result was established
by Kr\"{o}ger~\cite{K}. In either case, these are based on sharp bounds for the average of the first $k$ eigenvalues of the Laplacian, namely,
\begin{equation}\label{blyineq}
  \frac{\ds 1}{\ds k}\dsum_{j=0}^{k-1}\mu_{j}(\Omega) \leq  \frac{4 \pi^2
  d}{d+2}\left(\fr{k}{\wv d|\Omega|}\right)^{2/d}
  \leq \frac{\ds 1}{\ds k}\dsum_{j=1}^{k} \lambda_{j}(\Omega).
\end{equation}
From these inequalities and an estimate in \cite{K} it follows that, for individual eigenvalues,
\[
  \lambda_k(\Omega) \geq \frac{4 \pi^2 d}{d+2}\left(\fr{k}{\wv d|\Omega|}\right)^{2/d},
\] 
and
\[
  \mu_{k}(\Omega) \leq  4\pi^2 \left( \frac{d+2}{2}
  \right)^{2/d}\left(\fr{k}{\wv d|\Omega|}\right)^{2/d},
\]
which both fall short of~\eqref{PolyaConj}.

Note that inequalities \eqref{eq:strict} lead naturally to a strenghtening of
P\'olya's conjecture, which we also investigate.
\begin{open}[Strong P\'olya's conjecture]
  For all $\Omega\subset \R^d$ and all $k \in \N$,
 \begin{equation}\label{StrongPolyaConj}
  \mu_k(\Omega) <  \frac{4 \pi^2}{(\omega_d |\Omega|)^{2/d}} k^{2/d} < \lambda_k(\Omega).
 \end{equation}
\end{open}\vspace{6pt}

As mentioned above, the first two eigenvalues are known to satisfy the strong P\'{o}lya's inequalities since their extremal values are known. However,
for higher eigenvalues and although some conjectures do exist, there are no other situations where the extremal values are known.
Furthermore, numerical optimisations carried out within the last fifteen years by different researchers using different methods have made it clear
that not much structure at this level is to be expected in the mid-frequency range, in the sense that extremal sets are not described in
terms of known functions --- see~\cite{oude,anfr1} for the Dirichlet and~\cite{anfr1} for the Neumann problems respectively; see also~\cite{anfr3,boou}
for the same problem but with a perimeter restriction. In the planar case, it
has also been shown that, except for the first four eigenvalues,
the Dirichlet extremal domains are never balls or unions of balls~\cite{berg}. Recently, it has been shown that the Faber--Krahn inequality may
be used to extend the range of low Dirichlet eigenvalues for which P\'{o}lya's conjecture holds~\cite{fre2}. For instance, in dimensions three and larger,
eigenvalues up to $\lambda_{4}$ also satisfy P\'{o}lya's conjecture, with the number of eigenvalues which may be shown to do so by this method
growing exponentially with the dimension.

These findings prompted the study of what happens at the other end of the spectrum, in the high-frequency regime, in the
hope that some structure could be recovered there. The first of such results proved that, when restricted to the particular case of rectangles,
extremal domains converge to the square as $k$ goes to infinity~\cite{anfr2}. In other words, they converge to the domain with minimal perimeter
among all of those in the class of rectangles with fixed area, and indeed, just like with the first eigenvalue, the geometric isoperimetric
inequality plays a role in the proof. This was followed by an extension of these results to higher-dimension rectangles in both the
Dirichlet and Neumann cases~\cite{berburgit,bergit,gitlar,mars}. In the case of general planar domains with a perimeter restriction, it
was shown in~\cite{bufr} that extremal sets converge to the disk with the same perimeter as $k$ goes to infinity, thus again displaying
convergence to the geometric extremal set. Some results regarding existence of convergent subsequences within classes of convex domains
and under a measure restriction were also obtained in~\cite{lars}.

The connection between the problem of determining extremal domains for the $k^{\rm th}$ eigenvalue and P\'{o}lya's conjecture was established
in 2014 by Colbois and El Soufi~\cite{CE}. There they showed that the sequences of extremal values $\left(\lambda_{k}^{*}\right)^{d/2}$
(Dirichlet) and $\left(\mu_{k}^{*}\right)^{d/2}$ (Neumann) are subadditive and superadditive, respectively. As a consequence of
Fekete's lemma, both sequences $\lambda_{k}^{*}/k^{2/d}$ and $\mu_{k}^{*}/k^{2/d}$ are convergent as $k$ goes to infinity and, furthermore,
P\'{o}lya's conjecture is seen to be equivalent to
\[
 \lim_{k\to\infty} \fr{\lambda_{k}^{*}}{k^{2/d}} = \fr{4\pi^2}{\left(|\Omega|\omega_{d}\right)^{2/d}} \mbox{ and }
  \lim_{k\to\infty} \fr{\mu_{k}^{*}}{k^{2/d}} = \fr{4\pi^2}{\left(|\Omega|\omega_{d}\right)^{2/d}},
\]
in the Dirichlet and Neumann cases, respectively.

A major obstacle in attacking the general P\'olya's conjecture is that it is not even known if there exists an open domain
minimising $\lambda_k$ or maximising $\mu_k$ for $k\ge 3$ under volume
constraint. This prevents one from using properties of the minimisers
to argue in favor of the conjecture. Our aim will be to
restrict ourselves to the study of classes of domain within which we are able to
show existence of extremisers, but within which the subadditivity and
superadditivity results of Colbois and El Soufi still hold. Note that
subadditivity or superadditivity for the optimal eigenvalues do not hold for all
families of domains -- if we take as a family of domains rectangles of unit area, the
extremisers always exist but the optimal Dirichlet eigenvalues are 
$\lambda_{1}^{*} = 2\pi^2$, $\lambda_{3}^{*} = 5\pi^2$ and 
$\lambda_{4}^{*}=35\pi^2/(2\sqrt{6})\approx 7.144 \pi^2$, see~\cite{anfr2}.

\subsection{Suitable families of domains}

Before stating our results, let us define precisely the class of domains under
consideration in this paper. Given $r \in (0,\infty)$ and $\Omega \subset \R^d$,
we denote by $r\Omega$ any subset of $\R^d$ obtained from $\Omega$ as a result
of a homothety with scale factor $r$ and an isometry. 

\begin{defi} \label{defi:fam}
  Let $\Omega_1,\dotsc,\Omega_n$ be bounded, connected, open subsets
of $\R^d$. We denote
  \begin{equation}
    \CR:= \CR(\Omega_1,\dotsc,\Omega_n) := \set{\,\bigsqcup_{i \in I} r_i
      \Omega_{n_i} \, : \,  I \text{ countable}, n_i \in \set{1,\dotsc,n}, \sum_{i \in I}r_i^d < \infty \,
  }.
  \end{equation}
  The sets $\Omega_1,\dotsc,\Omega_n$ are called the generators for $\CR$. 
  The above notation is to be understood in the sense that all sets $\Upsilon \in
  \CR$ are subsets of $\R^d$ all of whose connected components are of the form
  $r_i \Omega_{n_i}$ for $i \in I$. We denote by $\nu(\Upsilon)$ the number of
  connected components of $\Upsilon$, by $\abs \Upsilon$ its volume and we
  slightly abuse notation by denoting by $\abs{\del \Upsilon}$ the $(d-1)$-dimensional Hausdorff measure of the boundary.
   We also
  observe that the family $\CR$ is closed under disjoint union and homothety, up
  to rearrangement. Whenever the Neumann eigenvalue problem is discussed, it is also assumed
  the generators have Lipschitz boundary.
\end{defi}
One particular instance of this type of families, namely, those generated by rectangles, was used recently
to study the possible asymptotic behaviour of extremal sets in the case of Robin boundary conditions~\cite{frke}.

The following elementary facts about scaling properties of volumes and
eigenvalues will be used repeatedly in this paper:
\begin{itemize}
  \item $\abs{r \Upsilon} = r^d \abs \Upsilon$;
  \item $\abs{r \del \Upsilon} = r^{d-1} \abs{\del \Upsilon}$;
  \item $\lambda_k(r\Upsilon) = r^{-2} \lambda_k(\Upsilon)$;
  \item $\mu_k(r\Upsilon) = r^{-2} \mu_k(\Upsilon)$;
\end{itemize}
It is easy to see from the first two points that the generator $\Omega_j$ minimising the
isoperimetric ratio among $\Omega_1,\dotsc,\Omega_n$ also does so in $\CR$. The
first, third and four bullet points imply that the quantities
$\lambda_k(\Upsilon)^{d/2}\abs \Upsilon$ and $\mu_k(\Upsilon)^{d/2} \abs
\Upsilon$ are invariant by homothety.
\begin{defi}
We define
\begin{equation}
 \lambda_k^*(\CR) = \inf_{\substack{\Upsilon \in \CR \\ |\Upsilon| \le 1}} \lambda_k(\Upsilon)
\end{equation}
and
\begin{equation}
 \mu_k^*(\CR) = \sup_{\substack{\Upsilon \in \CR \\ |\Upsilon| \ge 1}} \mu_k(\Upsilon).
\end{equation}
We shall say that a domain $\Upsilon \in \CR$ is a \textit{minimiser for} $\lambda_k^*(\CR)$ or that it \textit{realises} $\lambda_k^*(\CR)$ if
$|\Upsilon| \le 1$ and if $\lambda_k(\Upsilon) = \lambda_k^*(\CR)$. Similarly, a domain can be a \textit{maximiser} for $\mu_k^*(\CR)$ or it \textit{realises}
$\mu_k^*(\CR)$. Note that an extremiser necessarily verifies $|\Upsilon| =1$. \vspace{6pt}
\end{defi}

In Section \ref{sec:suitable}, we show that these families $\CR$ of domains are suitable for the study of asymptotic
eigenvalue optimisation. By suitable, we understand that for every $k$, there
exists $\Upsilon \in \CR$ realising the extremal eigenvalues, and that the results
of \cite{WK,PRF,CE} describing the extremal eigenvalues and their associated extremisers
still hold within the families $\CR$. Existence of the extremisers is proved in
Theorems \ref{existence_min} and \ref{existence_max}. 

The properties of extremal eigenvalues and their associated extremisers are the
subject of Theorems \ref{subadditivity}--\ref{prf}. They rely on the fact that
two properties are needed for the proofs of these theorems : closedness under
homotheties, and under disjoint unions. Of specific use is Corollary \ref{BLY},
which says that it is sufficient to study the limit of the sequence of optimal
eigenvalues if one wants to get universal bounds within a family $\CR$.

\subsection{A trichotomy for P\'olya's conjecture}

In Section \ref{sec:trichotomy}, we restrict our search to families $\CR$ generated by a single
domain $\Omega$. There is no loss of generality here : we will first show that
if P\'olya's conjecture holds
within two families $\CR(\Omega_1)$ and $\CR(\Omega_2)$ in either its standard
or strong form, then
it also holds in $\CR(\Omega_1,\Omega_2)$. 

Our aim is to characterise the structure of the set of optimisers in
$\CR(\Omega)$ depending on whether P\'olya's conjecture holds or fails in
$\CR$. This
gives, in principle, a way to investigate the conjecture for a given domain.
We note that P\'olya's conjecture remains open for any domain that does not tile
$\R^d$, notably in the case of the ball, even though we have explicit formulae for the
eigenvalues. 

Our main theorem is as follows.

\begin{thm} \label{thm:trichotomy}
  The following trichotomy holds : either
  \begin{enumerate}
    \item the generator $\Omega$ realises $\lambda_k^*(\CR)$ infinitely often and P\'olya's conjecture for Dirichlet eigenvalues holds in $\CR(\Omega)$;
    \item the generator $\Omega$ realises $\lambda_k^*$ only finitely many times, P\'olya's conjecture for Dirichlet eigenvalues holds in $\CR(\Omega)$ and, for
      infinitely many $k \in \N$,
      \begin{equation}
        \frac{\lambda_k^*(\CR)^{d/2}}{k} = \frac{(2 \pi)^d}{\omega_d},
      \end{equation}
      or
    \item the generator $\Omega$ realises $\lambda_k^*$ only finitely many
      times, P\'olya's conjecture for Dirichlet eigenvalues does not hold in $\CR(\Omega)$ and, for
      infinitely many $k \in \N$,
      \begin{equation}
        \frac{\lambda_k^*(\CR)^{d/2}}{k} = \inf_k \frac{\lambda_k^*(\CR)^{d/2}}{k}.
      \end{equation}
  \end{enumerate}
  The same trichotomy holds replacing all instances of Dirichlet with Neumann, and of $\lambda$ with $\mu$. 
\end{thm}

In Theorem \ref{smallo}, we furthermore obtain an indication of when $\Omega$ can realise $\lambda_k^*(\CR)$ or $\mu_k^*(\CR)$ infinitely often. Namely, we show that as
soon as there exists a subsequence $\set{k_n}$ such that
the number of connected components of the domain realising $\lambda_{k_n}^*$,
respectively $\mu_{k_n}^*$ has slower than linear growth, then $\Omega$ realises $\lambda_k^*$, respectively $\mu_k^*$ infinitely often. This, in combination with Lemmas \ref{wk}
and \ref{prf} allows us to understand the propagation of extremal domains in
$\CR$ as $k \to \infty$.

Finally, when the generator $\Omega$ satisfies the two-term Weyl law
\eqref{twotermweyldir} or \eqref{twotermweylneu}, we obtain
the following list of equivalences with the strong P\'olya conjecture
\begin{thm}
  \label{thm:tfae}
  Suppose that $\Omega \subset \R^d$ is such that the two-term Weyl law
  \eqref{twotermweyldir} holds.
 Let $\Omega_k^* = \bigsqcup_{i \in \N} r_{i,k} \Omega$ be a sequence of
 domains realising $\lambda_k^*(\CR)$.
 Suppose that $\abs{\Omega_k^*} = 1$ and $r_{i,k} > r_{j,k}$ whenever $i < j$.  
 The following are equivalent :
 \begin{enumerate}
   \item The strong P\'olya conjecture for Dirichlet eigenvalues holds in $\CR(\Omega)$.
  \item The largest coefficient $r_{1,k} \to 1$ as $k \to \infty$.
  \item The largest coefficient $r_{1,k} \to 1$ along a subsequence.
 \end{enumerate}
 The same equivalence hold replacing all instances of Dirichlet with Neumann,
 $\lambda$ with $\mu$, and the two-term Weyl law \eqref{twotermweyldir} with
 \eqref{twotermweylneu}.
\end{thm}
Comparing those equivalent statements to the trichotomy in Theorem
\ref{thm:trichotomy}, it is clear that if the strong P\'olya conjecture holds,
$\Omega = \Omega_k^*$ infinitely often. On the other hand, if $\Omega =
\Omega_k^*$ infinitely often, it is the case that $r_{1,k} \to 1$ along a
subsequence. Theorem \ref{thm:tfae} indicates that $\Omega$ satisfying a
two-term Weyl law implies that the strong P\'olya conjecture for $\CR(\Omega)$
is equivalent to weaker statements than those needed to imply P\'olya's conjecture in Theorem \ref{thm:trichotomy}. 

\subsection{Density lower bounds for Dirichlet eigenvalues}

In the paper \cite{urak}, Urakawa obtained a lower bound for Dirichlet eigenvalues
in terms of the lattice packing density of a domain $\Omega$. As an application
of our construction, we obtain in Section \ref{sec:packing} similar results for the asymptotic packing density defined as follows.

Given a set $\Omega$ and $n \in N$, we define the
\textit{$n$-th propagation of $\Omega$} as the set
\begin{equation}
 \Omega^{(n)} = \bigsqcup_{\ell = 1}^n \, \frac{1}{n^{1/d}} \,  \Omega \, .
\end{equation}

\begin{defi}\label{packing}
Given two bounded domains $\Omega$ and $V$ with volume $1$, an integer $n \in \N$
and a real number $\rho \in (0, 1]$, a \emph{packing of $\Omega^{(n)}$ into $V$
of density $\rho$} is an isometric
quasi-embedding $f : \overline{\Omega^{(n)}} \to \rho^{-1/d} V$.
Here, we call a map a \emph{quasi-embedding} if it is injective on the interior of its domain. Note furthermore that $\Omega$, and hence any element in $\CR$, is canonically equipped with a Riemannian metric. The term \emph{isometry} is to be understood as “preserving Riemannian metrics”. 

An \emph{asymptotic packing of $\Omega$ into $V$} is a triple
$P = \{(n_i, \rho_i, f_i)\}_{i \in \N}$ where $\{n_i\}_{i \in \N}$ is a 
strictly increasing sequence of integers, $\{\rho_i\}_{i \in \N} \subset (0, 1]$
converges to the \emph{asymptotic density} $\rho_P \in (0, 1]$ and each $f_i$ is
a packing of $\Omega^{(n_i)}$ into $V$ of density $\rho_i$.

The \emph{packing number or packing density of $\Omega$ into $V$} is
\begin{equation} 
\rho_{\Omega, V} \, = \, \sup \, \{ \, \rho_P \, | \, \mbox{ $P$ is an asymptotic packing of $\Omega$ into $V$ } \, \} \; .
\end{equation}
The \emph{packing number or packing density of $\Omega$} is
\begin{equation} 
\rho_{\Omega} \, = \, \sup \, \{ \, \rho_{\Omega, V} \, | \, \mbox{ $V$ is a bounded domain with volume $1$ } \, \} \;  .
\end{equation}

\begin{defi}\label{tile}
A domain $D \subset \R^d$ is a \emph{tile} or is said \emph{to tile $\R^d$} if there is an isometric quasi-embedding $F : \sqcup_{i \in \N} \bar{D} \to \R^d$, called the \emph{tiling}, which is surjective.
\end{defi}\vspace{6pt}

\end{defi}\vspace{6pt}
\begin{rem}
 The lattice packing density of Urakawa \cite{urak} is always smaller or equal
 to this packing density, as it is equivalent to considering only $V$ that are
 parallelepipeds, as well as having $P$ constrained more strictly. It is not
 hard to find examples of concave, simply connected domains that have a higher
 asymptotic packing density than their lattice packing density.
\end{rem}
We obtain the following theorem for a lower bound on Dirichlet eigenvalues in
terms of this asymptotic density.

\begin{thm} \label{thm:packing}
  For every $\Omega \subset \R^d$ open and bounded, with $\abs \Omega = 1$, the
  lower bound
  \begin{equation}
  \inf_k \frac{\lambda_k^*(\CR(\Omega))^{d/2}}{k} \ge \rho_\Omega
  \frac{(2\pi)^d}{\omega_d}
  \end{equation}
  holds.
\end{thm}
Obviously, the previous Theorem allows us to recover P\'olya's theorem as a
corollary.
\begin{cor}[P\'olya \cite{poly}] \label{ThmPolya}
If $\Omega$ tiles $\R^d$, then P\'olya's conjecture holds for any domain in $\CR(\Omega)$.
\end{cor}

\begin{proof}
  If $\Omega$ tiles $\R^d$, then $\rho_{\Omega} =1$ (see Proposition
  \ref{tiling}). Then Theorem \ref{thm:packing} implies the result.
\end{proof} \vspace{6pt}

We also obtain the following strengthening of P\'olya's theorem for domains that
are said to simply tile $\R^d$ and for which the two-term Weyl law
\eqref{twotermweyldir} holds, in which case the strong P\'olya conjecture holds.

\begin{defi}\label{stile}
A domain $\Omega$ is a \emph{simple tile} or is said \emph{to simply tile
$\R^d$} if there is a domain $V \subset \R^d$, called a fundamental domain, with volume $1$ and an asymptotic
packing $P = \{(n_i, 1, f_i)\}_{i \in \N}$ of $\Omega$ into $V$ with constant packing density $1$.
\end{defi}\vspace{6pt}

\begin{thm} \label{polya_stile}
If $\Omega$ simply tiles $\R^d$ with fundamental domain $V$ satisfying the two-term Weyl law
\eqref{twotermweyldir}, then $\Omega$ realises $\lambda_k^*(\CR(\Omega))$
infinitely often and satisfies the strong P\'olya conjecture. The same holds
for Neumann eigenvalues, if $V$ satisfies \eqref{twotermweylneu} instead.
\end{thm}

 \subsection{Computational results}
In Section \ref{sec:comp}, we investigate numerically the set of extremisers for
Dirichlet eigenvalues within
families $\CR$ generated by the disk, the square, and a rectangle with aspect
ratio $5$. We chose these domains to see if the markers for the
P\'olya conjecture differed between the rectangles, for which the conjecture is
known to hold, and the disk, for which it's not. In all four cases, we look for
extremisers up to eigenvalue rank 66 000. 

We investigate the number of connected components of the extremising set, in view of Theorem \ref{smallo}. In
all the cases we are studying, we see that this number is bounded by $5$, up to
rank 66 000. Recall that for P\'olya's conjecture to hold, we only need for a
subsequence of the extremisers to have their number of connected components
bounded.

We also investigate the asymptotic log-density of the number of times the generator
can be $\Omega_k^*$. For a set $J \subset \N$, we define its counting function
as
\begin{equation}
  N_J(x) := \#\set{j \in J : j \le x}
\end{equation}
and its log-density as
\begin{equation}\label{eq:logdensity}
F_J(x) := \frac{\log(N_J(x))}{\log x}.
\end{equation}
We have that for every $\eps > 0$,
\begin{equation}
  \lim_{x\to \infty} F_J(x) = \alpha > 0 \quad \Leftrightarrow \quad N_J(x) \ge
  x^{\alpha - \eps}
\end{equation}
for $x$ large enough. In particular, for $J$ the set of ranks $k$ for which the
generator realises $\lambda_k^*$, $\lim_{x \to \infty} F_J(x) = \alpha > 0$
implies that the cardinality of $J$ is infinite. The log-density in all cases we
investigated seemed to converge quite quickly to a constant greater than $0.8$,
albeit not the same constant for the disk and the various rectangles. It would
be an interesting line of investigation to understand the geometric properties
that influence the value of this constant. 

\subsection*{Acknowledgements}
We would like to thank Iosif Polterovich for useful discussions. P.F. was partially supported by the Funda\c c\~{a}o para a Ci\^{e}ncia
e a Tecnologia (Portugal) through project PTDC/MAT-CAL/4334/2014. J.L. was partially supported by ESPRC grant
EP/P024793/1 and NSERC postdoctoral fellowship. J.P. was partially supported by
the NSERC Alexander-Graham-Bell scholarship.

\section{Eigenvalue optimisation within a family} \label{sec:suitable}

Recall that for $\Omega_1,\dotsc,\Omega_n$, each of volume 1, we investigate the family of domains
\begin{equation}
  \CR(\Omega_1,\dotsc,\Omega_n) := \set{\,\bigsqcup_{i \in I} r_i \Omega_{n_i} \, : \,  I \text{ countable}, n_i \in \set{1,\dotsc,n}, \sum_{i \in I}r_i^d < \infty \, }.
\end{equation}

Our first two results concern the existence of eigenvalue extremisers in this restricted collection $\CR$. \vspace{0pt}

\begin{lem} \label{existence_min}
 For all $k$, there exists a domain $\Omega_k^* \in \CR$ of volume $1$ such that
 \begin{equation}
  \lambda_k(\Omega_k^*) = \lambda_k^*(\CR)
 \end{equation}
 and for any minimising domain $\nu(\Omega_k^*) \le k$.  
\end{lem}

\begin{proof}
 Fix $k \ge 1$. For any $j \in \N \cup \set \infty$, denote
 \begin{equation}
  \lambda_k^{(j)} = \inf \set{\lambda_k(\Upsilon) \,  : \, \Upsilon \in \CR, \, |\Upsilon| \le 1, \,  \nu(\Upsilon) = j}.
 \end{equation}
 
\noindent Of course, $\lambda_k^*(\CR) = \inf_{j} \,  \lambda_k^{(j)}$.

Our first step is to show that if $j > k$, then $\lambda_k^{(j)} \ge \lambda_k^{(l)}$ for some $l \le k$; It follows in particular that the previous infimum is a minimum.

The argument for this first step will follow the proof of \cite[Lemma 8]{BI}.
Indeed, consider $\Upsilon = \bigsqcup_{i \in I} r_i \Omega_{n_i} \in \CR$ with
$|\Upsilon| =1$ and $\nu(\Upsilon) = j$. Suppose without loss of generality that
\begin{equation} 
  \lambda_1(r_j\Omega_{n_j}) \le \lambda_1(r_{j'}\Omega_{n_{j'}})
  \text{ whenever } j \le j'
\end{equation}
Let
\begin{equation}
  l = \min\big\{k, \max \set{m : \lambda_1(r_m \Omega_{n_m}) \le \lambda_k(\Upsilon)}\big\} \le k,
\end{equation}
and
\begin{equation}
  \tilde \Upsilon = r_1 \Omega_{n_1} \sqcup \dotso \sqcup r_l \Omega_{n_l}.
\end{equation}
Note that if $\nu(\Upsilon) = \infty$, $m$ is still finite since $r_j \to 0$ as
$j \to \infty$, and observe that $\lambda_k(\tilde \Upsilon) \le \lambda_k(\Upsilon)$. Since
$|\tilde \Upsilon| \le 1$, we can dilate it to a set $\hat{\Upsilon}$ of volume $1$ whose eigenvalues are all smaller than the ones of $\tilde{\Upsilon}$, so that $\lambda_k(\hat{\Upsilon}) \le \lambda_k(\Upsilon)$. Taking the infimum of this inequality over all appropriate sets $\Upsilon$ and recalling that $\nu(\hat{\Upsilon}) = l \le k < j = \nu(\Upsilon)$, we get indeed
\begin{equation}
 \lambda_k^{(l)} \le \lambda_k^{(j)}.
\end{equation}
We therefore deduce that
\begin{equation}
 \lambda_k^* = \min_{1 \le j \le k} \lambda_k^{(j)}.
\end{equation}\vspace{0pt}

Our second step is to show that for every $1 \le j \le k$, either there exists a minimiser $\Upsilon^{(j)} \in \CR$ for $\lambda_k^{(j)}$ or $\lambda_k^{(j)} \ge \lambda_k^{(j-1)}$.

The statement is obvious for $\lambda_k^{(1)}$, as there is only a finite number
of set, namely $\Omega_1,\dotsc,\Omega_n$ to verify. For $j > 1$, consider a minimising sequence 
\begin{equation}
  \Upsilon^{(j)}_p = \bigsqcup_{i = 1}^j r_{i,p} \Omega_{n_{i,p}}
 \end{equation}
 of sets in $\CR$ which can all be taken to have volume $1$, \textit{i.e.}
\begin{equation}
 \lambda_k^{(j)} = \lim_{p \to \infty} \,  \lambda_k(\Upsilon_p^{(j)}) \, .
\end{equation}
Assume without loss of generality $1 > r_{1,p} \ge \dots \ge r_{j,p}$ for each
$p$. If $r_{j,p} \to 0$ as $p \to \infty$, then for $p$ large enough,
$\lambda_1(r_{j,p}\Omega_{n_{j,p}}) \ge \lambda_k^{(j)}$. This implies that
$\lambda_k(\Upsilon_p^{(j)} \setminus r_{j,p} \Omega_{n_{j,p}}) =
\lambda_k(\Upsilon_p^{(j)})$ but $\nu(\Upsilon_p^{(j)} \setminus r_{j,p} \Omega) = j
- 1$, hence $\lambda_k^{(j)} \ge \lambda_k^{(j-1)}$. If $r_{j,p} \not\to 0$ as
$p \to \infty$, then the set $\{r_{i,p}\}_{1 \le i \le j, p \in \mathbb{N}}$
belongs to a compact interval $[ \eps, 1 - \eps] \subset (0,1)$. For
every $1 \le i \le j$ let $(r_i^{(j)},n_i^{(j)})$ be an accumulation point of
$\{(r_{i,p},n_{i,p})\}_{p \in \mathbb{N}}$, then set 
\begin{equation}
  \Upsilon^{(j)} = \bigsqcup_{i = 1}^j r_{i}^{(j)} \Omega_{n_i^{(j)}} \in \CR.
\end{equation}
By continuity of the $k$-th eigenvalue and of the volume as functions of the variables $r_{1}, \dots, r_j$, the set $\Upsilon^{(j)}$ has volume $1$ and verifies $\lambda_k(\Upsilon^{(j)}) = \lambda_k^{(j)}$. 

We proved that there is a set of indices $J \subseteq \{1, \dots, k\}$ such that for all $j \in J$,
there exists a minimiser $\Upsilon^{(j)}$ of $\lambda_k^{(j)}$, whereas $\lambda_k^{(i)} \ge \min_{j \in J} \lambda_k^{j}$ for all $i \notin J$.
Therefore,
\begin{equation}
  \lambda_k^* = \min_{1 \le j \le k} \lambda_k^{(j)},
\end{equation}
is realised by the set $\Omega_k^* := \Upsilon^{(j)}$ for any (say, the smallest) index $j$ realising the previous minimum, thus completing the proof.
\end{proof} \vspace{6pt}

We now show the equivalent lemma for Neumann eigenvalues.

\begin{lem} \label{existence_max}
 For all $k \ge 1$, there exists a domain $\Omega_k^* \in \CR$ such that
 \begin{equation}
  \mu_k(\Omega_k^*) = \mu_k^*(\CR)
 \end{equation}
 and for any maximising domain $\nu(\Omega_k^*) \le k$.  
 \end{lem}

\begin{proof}
The first step of this proof is easier in the setting of Neumann eigenvalues. Indeed, no maximising sequence $\set{\Upsilon_n}$ for $\mu_k^*(\CR)$ can have $\nu(\Upsilon_n) > k$ infinitely often, since $\nu(\Upsilon) > k$ implies immediately $\mu_k(\Upsilon) = 0$.
 
For the second step, since the supremum for $\mu_k^*(\CR)$ is taken over domains of volume larger or equal to $1$, we need to verify both that no connected component of a maximising sequence converges to $0$ and that none grows unbounded. This last possibility is easily excluded by restricting our attention to maximising sequences of domains which all have volume $1$. However, showing that no connected component has volume converging to $0$ is subtler.
 
 For any $\eps$, consider the set
 \begin{equation}
  \CM_\eps = \set{\Upsilon \in \CR : |\Upsilon| \ge 1 \text{ and } \mu_k(\Upsilon) > \mu_k^*(\CR) - \eps)}.
 \end{equation}
One can restrict the search for a maximising sequence to any $\CM_\eps$ rather than $\CR$. We now show that for $\eps$ small enough, there exists some $\delta$ such that for all $\Upsilon \in \CM_\eps$, all connected components of $\Upsilon$ have volume greater than $\delta$. Using the same compacity argument as with Dirichlet eigenvalues, this is sufficient to obtain the existence of a maximiser.

Suppose that such a $\delta$ does not exist. Hence, there is a maximising sequence
\begin{equation}
  \Upsilon_p = \bigsqcup_{i = 1}^{q}r_{i,p} \Omega_{n_{i,p}}
\end{equation}
with the following properties.
\begin{itemize}
 \item For all $p$, the number of connected components $q$ is smaller than $k$.
 \item Arranging $r_{1,p} \le r_{2,p} \le \dotso \le r_{q,p}$, we have that
   $r_{1,p} \to 0$ as $p \to \infty$.
 \item The eigenvalues $\mu_k(\Upsilon_p)$ increase and converge to $\mu_k^*(\CR)$
   as $p \to \infty$.
\end{itemize}
We will write $\Upsilon_p = r_{1,p} \Omega_{n_1,p} \cup \Xi_p$, each of them having volume
$r_{1,p}^d$ and $1 - r_{1,p}^d$ respectively. 

From \cite{K}, we know that there is a constant $C_k$ such that for all $k$,
$\mu_k(\Upsilon) < C_k$. There is an $r_0$ such that for all $l$,
$r_0^{-2}\mu_1(\Omega_l) \ge C_k$. For $p$ large enough so that $r_{1,p} < r_0$,
we have $r_{1,p}^{-2}\mu_1(\Omega_{n_{1,p}}) > C_k$, hence $\mu_k(\Upsilon_p) =
\mu_{k-1}(\Xi_p)$.

For any $\eta \in (0,r_0)$, consider the following sequence of domains of volume $1$ in $\CR$:
\begin{equation}
 \tilde \Upsilon_p^{(\eta)} = \eta \Omega_1 \sqcup \left(\frac{1 - \eta^d}{1 -
 r_{1,p}^d}\right)^{1/d} \Xi_p.
\end{equation}
Without loss of generality, we have supposed $\eta < 1$.

For $p$ large and since $\eta < r_0$, 
\begin{equation}
\begin{aligned}
\mu_k(\tilde \Upsilon_{p}) &= \mu_{k-1}\left(\left(\frac{1 - \eta^d}{1 -
r_{1,p}^d}\right)^{1/d} \Xi_p\right) \\
&= \left(\frac{1 - r_{1,p}^d}{1 - \eta^d}\right)^{2/d} \mu_k(\Upsilon_p) \\
&= \frac{1}{(1 - \eta^d)^{2/d}} \mu_k(\Upsilon_p)\left(1 + \bigo{r_{1,p}}\right).
\end{aligned}
\end{equation}

Hence,
\begin{equation}
\begin{aligned}
 \mu_k(\tilde \Upsilon_p) - \mu_k(\Upsilon_p) &= \left(\frac{1 + \bigo{r_{1,p}}}{(1
 - \eta^d)^{2/d}} -  1 \right) \mu_k(\Upsilon_p) \\
 & \ge \frac{2\eta^d \mu_k(\Upsilon_1)}{d} \left(1 + \bigo{r_{1,p}}\right)
\end{aligned}
 \end{equation}
Since $\mu_k^*(\CR) > \mu_k(\tilde \Upsilon_p)$ this implies that if $\eps <
\frac{2\eta^d \mu_k(\Upsilon_1)}{d}$, then for $p$ large enough $\Upsilon_p \not \in
\CM_\eps$. Since every maximising sequence must eventually stay in
$\CM_\eps$ for any $\eps$, this contradicts that $\Upsilon_p$ was a maximising sequence. This in turn implies the lemma.
\end{proof}

Note that both of these proofs show existence but say nothing about uniqueness. Despite this possible lack of uniqueness, in this paper we shall write $\Omega_k^*$ to denote any extremiser of $\lambda_k$ or of $\mu_k$ on $\CR$.

\begin{lem}\label{subadditivity}
 The sequence
 \begin{equation}
 \set{ \lambda_k^*(\CR)^{d/2}}_{k \in \mathbb{N}}
 \end{equation}
is subadditive, that is for every $j_1, \dotsc j_p$ such that $j_1 + \dotso j_p = k$, we have
\begin{equation}
 \lambda_k^*(\CR)^{d/2} \le \lambda_{j_1}^*(\CR)^{d/2} + \dotso + \lambda_{j_p}^*(\CR)^{d/2}.
\end{equation}

\end{lem}

\begin{proof}
 The proof here follows that of \cite[Theorem 2.1]{CE}. Fix $k \ge 1$ and let
 $j_1, \dots, j_p \in \mathbb{N}$ be such that $j_1 + \dots + j_p = k$. By Lemma
 \ref{existence_min}, for each $1 \le q \le p$, there exists $\Omega_{j_q}^* \in \CR$ with volume $1$ such that
 \begin{equation}
   \lambda_{j_q}^*(\CR) = \lambda_{j_q}(\Omega_{j_q}^*).
 \end{equation}
Let
\begin{equation}
 \Upsilon_q := \left( \frac{\lambda_{j_q}^*(\CR)}{\lambda_k^*(\CR)}\right)^{1/2}
 \Omega_{j_q}^*,
\end{equation}
which implies that $\lambda_{j_q}(\Upsilon_q) = \lambda_k^*(\CR)$ and that 
\begin{equation}
|\Upsilon_q| = \left(\frac{\lambda_{j_q}^*(\CR)}{\lambda_k^*(\CR)}\right)^{d/2}.
\end{equation}
Define the domain
\begin{equation}
 \Upsilon = \bigsqcup_{q = 1}^p \Upsilon_q.
\end{equation}
Since the spectrum of a disjoint union is the union of the spectra, we have
\begin{equation}
 N(\lambda_k^*(\CR);\Upsilon) = \sum_{q = 1}^p N(\lambda_k^*(\CR);\Upsilon_q) =
 \sum_{q = 1}^p N(\lambda_{j_q}^*(\Upsilon_q);\Upsilon_q) \ge \sum_{q=1}^p j_q = k 
\end{equation}
where $N$ is the eigenvalue counting function
\begin{equation}
\label{eq:counting}
N(\lambda;\Upsilon) := \# \set{k : \lambda_k(\Upsilon) \le \lambda}.
\end{equation}
It follows that
$\lambda_k(\Upsilon) \le \lambda_k^*(\CR)$. Since $|\Upsilon|^{-1/d} \Upsilon$ has
volume $1$ we have $\lambda_k^*(\CR) \le \lambda_k \left( |\Upsilon|^{-1/d} \Upsilon
\right) =  \lambda_k(\Upsilon)|\Upsilon|^{2/d}$, thus
\begin{equation}
 |\Upsilon| \ge \left(\frac{\lambda_k^*(\CR)}{\lambda_k(\Upsilon)}\right)^{d/2} \ge 1, 
\end{equation}
whence
\begin{equation}
 1 \le \sum_{q=1}^p |\Upsilon_n| = \frac{1}{\lambda_k^*(\CR)^{d/2}}\sum_{q = 1}^p
 \lambda_{j_q}^*(\CR)^{d/2}.
\end{equation}
Multiplying both sides of this inequality by $\lambda^*_k(\CR)^{d/2}$ finishes the proof.
\end{proof}\vspace{6pt}

\begin{lem}\label{super-additivity}
 The sequence
 \begin{equation}
 \set{ \mu_k^*(\CR)^{d/2}}_{k \in \mathbb{N}}
 \end{equation}
is super-additive, that is for every $j_1, \dotsc j_p$ such that $j_1 + \dotso j_p = k$, we have
\begin{equation}
 \mu_k^*(\CR)^{d/2} \ge \mu_{j_1}^*(\CR)^{d/2} + \dotso + \mu_{j_p}^*(\CR)^{d/2}.
\end{equation}

\end{lem}
\begin{proof}
 Suppose on the contrary that there exist $j_1, \dots, j_p, k \in \N$ such that $j_1 + \dots + j_p = k$ and
 \begin{equation}
 \mu_k^*(\CR)^{d/2} < \mu_{j_1}^*(\CR)^{d/2} + \dotso + \mu_{j_p}^*(\CR)^{d/2},
 \end{equation}
 that is
  \begin{equation}
 1 \, < \,  \sum_{q=1}^p \left( \frac{ \mu_{j_q}^*(\CR)}{\mu_k^*(\CR)} \right)^{d/2} \, .
 \end{equation}
 From Lemma \ref{existence_max}, for every $1 \le q \le p$ there exists
 $\Omega^*_{j_q} \in \CR$ with volume $1$ such that $\mu_{j_n}(\Omega^*_{j_q}) =
 \mu^*_{j_q}(\CR)$. We set
 \begin{equation}
 \Upsilon \, = \, \bigsqcup_{q=1}^p \, \Upsilon_q \; \mbox{ where } \; \Upsilon_q \, =
 \, \left( \frac{\mu_{j_q}^*(\CR)}{\mu_k^*(\CR)}  \right)^{1/2} \, \Omega_{j_q}^* \, .
 \end{equation}
 It follows that $\mu_{j_q}(\Upsilon_q) = \mu_k^*(\CR)$ and that
 \begin{equation}
 |\Upsilon| \, = \, \sum_{q=1}^p \, |\Upsilon_q| =  \sum_{q=1}^p \left( \frac{
 \mu_{j_q}^*(\CR)}{\mu_k^*(\CR)} \right)^{d/2} > 1 \, .
 \end{equation}
 From this and since $|\Upsilon|^{-1/d}\Upsilon$ has volume $1$, we have
 \begin{equation}
\mu_k(\Upsilon) < |\Upsilon|^{2/d} \mu_k(\Upsilon) = \mu_k \left( |\Upsilon|^{-1/d} \Upsilon \right)  \le \mu_k^*(\CR) \, .
\end{equation}
Consequently $\mu_k(\Upsilon) < \mu_{j_q}(\Upsilon_n)$ for each $q$ and we deduce,
recalling that the spectrum of $\Upsilon$ is the reunion of the spectra of the
$\Upsilon_q$'s,
 \begin{equation}
 k+1 \le N(\mu_k(\Upsilon);\Upsilon) = \sum_{q = 1}^p N(\mu_k(\Upsilon);\Upsilon_q) \le
 \sum_{q = 1}^p j_q = k \, ,
 \end{equation}
 where the counting function is defined as in \eqref{eq:counting} but for Neumann eigenvalues. This contradiction yields the claim.
\end{proof}\vspace{6pt}

\begin{cor}\label{BLY}
We have
\begin{equation}
 L := \lim_{k \to \infty} \frac{\lambda_k^*(\CR)^{d/2}}{k} = \inf_k \frac{\lambda_k^*(\CR)^{d/2}}{k} > 0 
 \end{equation}
 and
 \begin{equation}
  +\infty > M := \lim_{k \to \infty} \frac{\mu_k^*(\CR)^{d/2}}{k} = \sup_k \frac{\mu_k^*(\CR)^{d/2}}{k} > 0 \; .
\end{equation}
\end{cor}

\begin{proof}
 For the Dirichlet case, that the limit exists and is equal to the infimum follows from Fekete's lemma applied to the subadditive and nonnegative sequence $a_k = \lambda_k^*(\CR)^{d/2}$. That the limit is positive is a consequence of the works of Berezin~\cite{B} and Li and Yau~\cite{LY} proving that
 \begin{equation}
  \frac{\lambda_k^{* \, d/2}}{k} \ge \left( \frac{d}{d+2}\right)^{d/2} \, \frac{(2 \pi)^d }{\omega_d}.
 \end{equation}
 For the Neumann case, that the limit exists in $\R$ and is equal to the supremum follows from Fekete's lemma applied to the super-additive and linearly bounded sequence $a_k = \mu_k^*(\CR)^{d/2}$, where the linear boundedness results from Kröger's estimate \cite{K}\footnote{In Kröger's article, Neumann eigenvalues are numbered starting with $1$ so that $\mu_1 = 0$.}
 \begin{equation}
  \frac{\mu_k^{* \, d/2}}{k} \le  \frac{d+2}{2} \, \frac{(2 \pi)^d }{\omega_d}  .
 \end{equation}
 That the limit is positive follows from $\mu_k(\Omega) \le \mu_k^*$ and from Weyl's asymptotic law
 \begin{equation}
 \lim_{k \to \infty} \frac{\mu_k(\Omega)^{d/2}}{k} =  \frac{(2 \pi)^d }{\omega_d}.
 \end{equation}
\end{proof}

P\'olya's conjecture can therefore be expressed as
\begin{equation}
L = \frac{(2 \pi)^d}{\omega_d} = M
\end{equation}
and thus reduces to finding a subsequence of extremisers $\Omega_k^*$ such that
\begin{equation}
 \lim_{k \to \infty} \frac{\lambda_k^*(\Omega_k^*)^{d/2}}{k} = \frac{(2 \pi)^d}{\omega_d} = \lim_{k \to \infty} \frac{\mu_k^*(\Omega_k^*)^{d/2}}{k}\; .
\end{equation}

The following lemma is an adaptation of a famous result of Wolf and Keller \cite{WK} to the class $\CR$. Our proof however differs somewhat from the original proof.
 
\begin{lem} \label{wk}
  For every $k \in \N$, and any $\Omega_k^*$ realising $\lambda_k^*(\CR)$, there exists a partition $j_1 + \dotso + j_p = k$ such that
  \begin{equation}
    \Omega_k^* = \bigsqcup_{q=1}^p \alpha_q \Omega_{j_q}^*.
  \end{equation}
  Here,
  \begin{equation}
    \alpha_q = \sqrt{\frac{\lambda_k^*(\CR)}{\lambda_{j_q}^*(\CR)}}.
  \end{equation}
  Furthermore,
  \begin{equation}
    \lambda_k^*(\CR)^{d/2} = \sum_{q=1}^p \lambda_{j_q}^*(\CR)^{d/2}.
  \end{equation}
\end{lem}
\begin{proof}
  If $\lambda_k^*$ is realised by one of the $\Omega_j$, we are done. Suppose it is not. By Lemma \ref{existence_min}, any minimiser for $\lambda_k$ has at most $k$ connected components. One also sees that the largest eigenvalue
smaller or equal to $\lambda_k^*(\CR)$ of each component has to be equal to $\lambda_{k}^*(\CR)$. If not it would be possible to decrease $\lambda_k^*(\CR)$ by shrinking slightly a component whose for which that's not the case at the
expense of expanding slightly the others.
 
In other words, if $\Omega_k^*$ is an optimal domain for $\lambda_k^*(\CR)$,
then each of its $p$ components ($p\leq k$) will have some eigenvalue rank $j_{q}$
 such that
 \[
  \Omega_k^* = \sqcup_{q=1}^{p}\Upsilon_q, \hspace*{1cm} \Upsilon_{q} =
  \alpha_{q}\Omega_{n_q},
 \]
 where
 \begin{equation}
   \sum_{q=1}^p \alpha_q^d = 1, \qquad \qquad \sum_{q=1}^p j_q = k,
 \end{equation}
 and
 \begin{equation}
   \lambda_{j_1}(\Upsilon_1) = \dotso = \lambda_{j_p}(\Upsilon_p) = \lambda_k^*(\CR).
 \end{equation}
 Furthermore, each of these $\Upsilon_q$ realises $\lambda_{j_q}^*$, otherwise it could be replaced by a domain who does while improving the eigenvalue.
The identities between the eigenvalues of the different components may now be written as

\[
  \alpha_{q}^{2} \lambda_{j_{p}}(\Omega_{n_p}) =
  \alpha_{p}^{2}\lambda_{j_{q}}(\Omega_{n,q}),\, q=1,\dotsc,p-1, 
\]
or
\[
  \alpha_{q}^{d} \lambda_{j_{p}}^{d/2}(\Omega_{n_p}) =
  \alpha_{p}^{d}\lambda_{j_{q}}^{d/2}(\Omega_{n_q}), \, q=1,\dotsc,p-1,.
\]
Summing up these identities for $j$ from $1$ to $p-1$,
\[
  \left(\dsum_{q=1}^{p-1} \alpha_{q}^{d}\right)
  \lambda_{j_{p}}^{d/2}(\Omega_{n_p}) =
  \alpha_{p}^{d}\dsum_{q=1}^{p-1}\lambda_{j_{q}}^{d/2}(\Omega_{n_1}).
\]
Hence
\[
 (1-\alpha_{p}^{d})\lambda_{j_{p}}^{d/2}(\Omega) = \alpha_{p}^{d}\dsum_{q =
 1}^{p-1}\lambda_{j_{q}}^{d/2}(\Omega_{n_q})
\]
and
\[
  \alpha_{p}^{d} =
  \fr{\lambda_{j_{p}}^{d/2}(\Omega_{n_p})}{\dsum_{q=1}^{p}\lambda_{j_{q}}^{d/2}(\Omega_{n_q})}.
\]
We finally obtain
\[
\begin{array}{lll}
  \lambda_{k}(\Omega_{k}^{*}) & = & \alpha_{p}^{-2}\lambda_{j_{p}}(\Omega_{n_p})\eqskip
& = & \left(\dsum_{q=1}^{p} \lambda_{j_{q}}^{d/2}(\Omega_{n_q})\right)^{2/d},
\end{array}
\]
yielding the desired result.
\end{proof}

A corresponding statement for Neumann eigenvalues is proved by Poliquin and Roy-Fortin \cite{PRF}
by closely mirroring Wolf and Keller's proof, and the result is recollected and somewhat
generalised by Colbois and El Soufi \cite{CE}. We include their proof in our formalism for completeness.

\begin{lem} \label{prf}
  For every $k \in \N$, and any $\Omega_k^*$ realising $\mu_k^*(\CR)$, there exists a partition $j_1 + \dotso + j_p = k$ such that
  \begin{equation}
    \Omega_k^* = \bigsqcup_{q=1}^p \alpha_q \Omega_{j_q}^*.
  \end{equation}
  Here,
  \begin{equation}
    \alpha_q = \sqrt{\frac{\mu_k^*(\CR)}{\mu_{j_q}^*(\CR)}}.
  \end{equation}
  Furthermore,
  \begin{equation}
    \mu_k^*(\CR)^{d/2} = \sum_{q=1}^p \mu_{j_q}^*(\CR)^{d/2}.
  \end{equation}
\end{lem}

\begin{proof}
  Once again, if $\mu_k$ is realised by one of the $\Omega_j$, we are done. A rather simple induction argument reduces the problem to the case $p=2$ and $\Omega_k^* = \Upsilon_1 \sqcup \Upsilon_2$ into two nonempty reunions of connected components, so that $|\Upsilon_1|, |\Upsilon_2| > 0$ and $|\Upsilon_1|+|\Upsilon_2| = |\Upsilon^*_k| =1$.

Choose $k+1$ of the $N(\mu^*_k(\CR), \Omega^*_k)$ lowest and linearly
independent eigenfunctions on $\Omega^*_k$, say $u_0, \dots, u_k$ ordered
according to their eigenvalues, in such a way that every eigenfunction with
eigenvalue strictly smaller than $\mu^*_k(\CR)$ is chosen and that every
eigenfunction is supported in either $\Upsilon_1$ or $\Upsilon_2$.\footnote{Recall
that $u_0$ is necessarily a locally constant function.} We have in particular
$\mu_k(u_k) = \mu_k^*(\CR) \ge \mu_k(\Omega) > 0$, where the last inequality
follows since $\Omega$ is connected. For every $0 \le l \le k$, the function
$u_l$ is not identically zero on at least one of the two $\Upsilon_q$'s; without lost of generality, assume that $u_k$ is not identically zero on $\Upsilon_1$. Notice that if the number of $u_l$'s which are not identically zero on $\Upsilon_1$ is $j_1 + 1$, then $\mu_{j_1}(\Upsilon_1) = \mu_k(u_k)$.

Since the spectrum of $\Omega^*_k = \Upsilon_1 \sqcup \Upsilon_2$ is the (ordered) union of the spectra of $\Upsilon_1$ and $\Upsilon_2$, and since the $u_l$'s span any eigenfunction on $\Omega^*_k$ with eigenvalue strictly smaller than $\mu^*_k(\CR)$, the number of $u_l$'s which are not identically zero on $\Upsilon_2$ is $j_2 = k - j_1$. Considering the $(j_2 + 1)$-th eigenfunction on $\Upsilon_2$ we get $\mu_{j_2}(\Upsilon_2) \ge \mu^*_k(\CR) > 0$; in particular $j_2 \ge 1$. We claim that in fact $\mu^*_k(\CR) = \mu_{j_2}(\Upsilon_2)$; to see this, suppose on the contrary that $\mu^*_k(\CR) < \mu_{j_2}(\Upsilon_2)$. Then consider any sufficiently small deformation $\Omega'$ (with volume $1$) of $\Omega^*_k$ obtained by contracting $\Upsilon_1$ to $\Upsilon'_1$ and dilating $\Upsilon_2$ to $\Upsilon'_2$, so as to have 
\begin{equation}
 \mu_{j_2-1}(\Upsilon'_2) < \mu_{j_2-1}(\Upsilon_2) \le \mu_{j_1}(\Upsilon_1) < \mu_{j_1}(\Upsilon'_1)   < \mu_{j_2}(\Upsilon'_2) < \mu_{j_2}(\Upsilon_2) \; .
\end{equation}
Hence $\mu_k(\Omega') = \mu_{j_1}(\Upsilon'_1)$ and thus $\mu_k(\Omega') > \mu^*_k(\CR)$. This contradicts the maximality of $\Omega^*_k$. As a result $\mu_{j_1}(\Upsilon_1) = \mu_{j_2}(\Upsilon_2) = \mu^*_k(\CR) > 0$. Since $\mu_j(D) > 0$ if and only if $j \ge \nu(D)$, we deduce $j_i \ge \nu(\Upsilon_i) \ge 1$. That we have a partition follows from $j_2 := k - j_1$.

We claim that the normalised domain $|\Upsilon_1|^{-1/d} \Upsilon_1$ realises $\mu_{j_1}^*(\CR)$. Suppose differently: There exists a maximiser $\Omega^*_{j_1}$ (with volume $1$) such that $\mu_{j_1}(|\Upsilon_1|^{-1/d} \Upsilon_1) \,  < \,  \mu_{j_1}(\Omega^*_{j_1}) = \mu^*_{j_1}(\CR)$, from which it follows that
\begin{equation}\label{ineqmu}
\mu^*_k(\CR) = \mu_{j_1}(\Upsilon_1) = |\Upsilon_1|^{-2/d} \mu_{j_1}(|\Upsilon_1|^{-1/d}\Upsilon_1) < |\Upsilon_1|^{-2/d} \mu_{j_1}^*(\CR) \, .
\end{equation}
Consider the domain
\begin{equation}
\tilde{\Omega} = \tilde{\Upsilon}_1 \sqcup \Upsilon_2 = \left( \frac{\mu_{j_1}^*(\CR)}{\mu^*_k(\CR)} \right)^{1/2}  \Omega^*_{j_1} \sqcup \Upsilon_2 \, .
\end{equation}
Equation \eqref{ineqmu} implies that its volume is strictly greater than $|\Upsilon_1||\Omega^*_{j_1}| + | \Upsilon_2 | = 1$. The $j_1+1$ first eigenvalues coming from $\tilde{\Upsilon}_1$ have eigenvalue at most $\mu^*_k(\CR)$, the $(j_1+1)$-th eigenvalue $\mu_{j_1}(\tilde{\Upsilon}_1)$ being equal to this value. Together with the same $j_2 = k - j_1$ eigenfunctions on $\Upsilon_2$ as before, we deduce that $\mu_k(\tilde{\Omega}) = \mu^*_k(\CR)$. Therefore the $(k+1)$-th eigenvalue of the normalised domain $|\tilde{\Omega}|^{-1/d} \tilde{\Omega}$ is strictly larger than $\mu^*_k(\CR)$, which is a contradiction to the maximality of $\Omega^*_k$. A similar argument implies that the normalised domain $|\Upsilon_2|^{-1/d} \Upsilon_2$ realises $\mu_{j_2}^*(\CR)$. Incidentally, $|\Upsilon_i| = (\mu^*_{j_i}(\CR)/\mu^*_k(\CR))^{d/2}$.

\end{proof}



\section{A Trichotomy}\label{sec:trichotomy}

In this section, we set out to prove Theorem \ref{thm:trichotomy}. Note that all
of the results of the previous sections have a Dirichlet and Neumann
version, where the only difference is that the inequalities are reversed. As
such, we will only prove the Dirichlet case of Theorem \ref{thm:trichotomy}, and
only state the corollaries in term of the Dirichlet eigenvalues. However, since
we rely only on the formal properties obtained in the previous section, all the
results also apply for Neumann eigenvalues, reversing the inequalities when
needed and changing the proofs \emph{mutatis mutandis}.
\vspace{6pt}

We start with the following proposition, allowing us to consider classes of domains $\CR$ generated by a single domain $\Omega$.

\begin{prop}
\label{prop:single}
 If P\'olya's conjecture holds within $\CR(\Omega_1)$ and $\CR(\Omega_2)$, then
 it holds within $\CR(\Omega_1,\Omega_2)$. The same is true of the strong
 P\'olya conjecture.
\end{prop}

It is clear that it is sufficient to show that if P\'olya's conjecture holds for
two domains $\Upsilon_1 \in \CR(\Omega_1)$ and $\Upsilon_2 \in \CR(\Omega_2)$,
then it holds for the disjoint union of these
two domains $\Upsilon_1 \sqcup \Upsilon_2$. This will rely on the following abstract lemma about superlinear sequences.

\begin{lem}
 Let $\set{a_k : k \in \N}$ and $\set{b_k : k \in \N}$ be two increasing sequences satisfying
 $$ a_k \ge \frac k A \qquad \text{and} \qquad b_k \ge \frac k B$$
 for some $A,B>0$.
 Denote $c_k$ the sequence obtained as the arrangement in increasing order of all elements in $\set{a_k} \sqcup \set{b_k}$, repeated with multiplicity. Then,
 \begin{equation} \label{eq:eqc} c_k \ge \frac{k}{A+B}.\end{equation}
 The same holds when all inequalities are replaced with strict inequalities.
\end{lem}
\begin{proof}
  Without loss of generality, we assume that $c_k = a_p$ for some $1 \le p \le
  k$. We distinguish two cases : $p = k$ and $1 \le p < k$. In the former
  situation, we have that
  \begin{equation}
    c_k = a_k \ge \frac k A > \frac{k}{A+B}.
  \end{equation}
  In the second case, it follows that $a_p \ge b_j$ for all $j$, $1 \le j \le
  k-p$. We then have
 \begin{align}
   \label{eq:str}
   \frac{k}{A + B} &= \frac{p + (k-p)}{A+B} \\
 &\le \frac{A a_p}{A+B}  + \frac{B b_{k-p} }{A+B} \\
 &\le a_p = c_k,
 \end{align} 
 where the last line holds from the fact that $a_p \ge
 \max\set{a_p,b_{k-p}}$, hence it is also greater than any convex combination of
 both. This concludes the proof, and it is readily seen that if the inequalities
 in the statement of the lemma were strict, then the second line in
 \eqref{eq:str} would be a strict inequality.
\end{proof}
To prove Proposition \ref{prop:single},
 apply the previous lemma with $a_k = \lambda_k(\Omega_1)^{d/2}$, $b_k =
 \lambda_k(\Omega_2)^{d/2}$, $a = \frac{\omega_d \abs{\Omega_1}}{(2\pi)^d}$, and
 $b = \frac{\omega_d \abs{\Omega_2}}{(2\pi)^d}$.

 Let us now define the set $J\subset \N$ of indices where the generator $\Omega$ realises $\lambda_k^*(\CR)$, that is
 \begin{equation}
   J:= \set{k \subset \N : \Omega = \Omega_k^*}.
 \end{equation}

\begin{prop} \label{Infinite}
Suppose $J$ is infinite, so that there exists a sequence $j_1 < j_2 < \dots \nearrow +\infty$ such that $\Omega = \Omega_{j_{n}}^*(\CR)$ for all $n$. Then P\'olya's conjecture is true for every $\Upsilon \in \CR$.
\end{prop}

\begin{proof}
 On the one hand, Weyl's law implies
 \begin{equation}
  \lim_{n \to \infty} \frac{\lambda_{j_n}(\Omega)}{j_n^{2/d}} =  \frac{4\pi^2}{ \omega_d^{2/d}} \, .
 \end{equation}
On the other hand, since $\Omega$ realises $\lambda_{j_n}^*(\CR)$ for every $n$, it follows from Corollary \ref{BLY}
\begin{equation}
 \lim_{n \to \infty} \frac{\lambda_{j_n}(\Omega)}{j_n^{2/d}} = \inf_{k}\frac{\lambda_k^*(\CR)}{k^{2/d}}= \frac{4\pi^2}{ \omega_d^{2/d}} \, .
\end{equation}
We therefore conclude that $\lambda_k(\Upsilon)^{d/2} \, k^{-1} \ge (2\pi)^d \omega_d^{-1}$ for every $\Upsilon \in \CR$ with volume $1$, which is P\'olya's conjecture.
\end{proof} \vspace{6pt}

The following theorem characterises when $J$ is finite.
\begin{thm} \label{smallo}
The set $J$ is finite if and only if there exists a constant $c$ such that for
all $k$, $\nu(\Omega^*_{k}) \ge ck$.
\end{thm}
\begin{proof}
  If $J$ is infinite, it is clear that such a constant $c$ does not exist. 
  Conversely, suppose that the set $J = \{ \, k \in \mathbb{N} \, : \, \Omega =
\Omega_{k}^*(\CR) \, \}$ is finite. This implies that any minimiser realising $\lambda_k^*(\CR)$ is of the form
\begin{equation}
  \Omega_k^* = \bigsqcup_{j \in J} \, \bigsqcup_{m=1}^{n_{k,j}} \, r_{k,j}
  \Omega_{j}^*.
\end{equation}
The number of connected components of $\Omega_k^*$ is
\begin{equation}
  \label{eq:cc}
  \nu(\Omega_{k}^*) =  \sum_{j \in J} n_{k, j},
\end{equation}
and referring to Lemma \ref{wk} we get
\begin{equation}
  \label{eq:ev}
  \lambda_k^*(\CR)^{d/2} = \sum_{j \in J} n_{k,j} \lambda_{j}^*(\CR)^{d/2}.
\end{equation}
Corollary \ref{BLY} states that there is a constant $c$ such that
$\lambda_{k}^*(\CR)^{d/2} \ge c' k$. Let $j' = \max J$, combining \eqref{eq:cc} and
\eqref{eq:ev} we obtain
\begin{equation}
  \begin{aligned}
    \nu(\Omega_k^*) &\ge \frac{1}{\lambda_{j'}(\Omega)^{d/2}}
    \sum_{j \in J} n_{k,j} \lambda_j^*(\CR)^{d/2} \\ 
    &\ge \frac{c'}{\lambda_{j'}(\Omega)^{d/2}} k.
  \end{aligned}
\end{equation}
The proof is completed by taking $c = c'\lambda_{j'}(\Omega)^{-d/2}$.
\end{proof}

Considering that all known results in the literature point to the validity of
P\'olya's conjecture, we are thus naturally led to the following, stronger,
conjecture. 

\begin{open} \label{subsequence}
For every domain $\Omega \subset \R^d$ there exists a subsequence
$\lambda_{k_n}^*(\CR(\Omega))$, with minimisers $\Omega_{k_n}^*$ such that
 \begin{equation}
  \nu(\Omega_{k_n}) = \smallo{k_n}.
 \end{equation}

\end{open}\vspace{6pt}

\noindent That this open problem is a potentially strictly stronger statement than P\'olya's conjecture
follows from this partial converse to \ref{Infinite} .\vspace{6pt}

\begin{prop} \label{Minimum}
 Suppose $J \subset \mathbb{N}$ is finite. Then,
 \begin{equation}
  \inf_{k} \frac{\lambda_k^*(\CR)^{d/2}}{k} = \min_{j \in J} \frac{\lambda_{j}(\Omega)^{d/2}}{j} \le \frac{(2\pi)^d}{\omega_d}.
 \end{equation}
\end{prop} \vspace{6pt}

Before starting with the proof, let us observe two things about this statement.
First, it means that $\inf_k \lambda_k^*(\CR)^{d/2}k^{-1}$ is realised. Second,
it means that if $\Omega$ is a minimiser in $\CR$ only for finitely many $k$'s
and if P\'olya's conjecture holds, then P\'olya's bound is attained since the
realised minimum of $\lambda_k^*(\CR)^{d/2}$ would be exactly $(2\pi)^dk\omega_d^{-1}$.

\begin{proof}
 Let
 \begin{equation}
  L' = \min_{j \in J} \, \frac{\lambda_j(\Omega)^{d/2}}{j} .
 \end{equation}
It exists as $J$ is finite, and $L' \ge L$. For any $k \not \in J$, a set which
realises $\lambda_k^*(\CR)$ necessarily has several connected components. It
results from Lemma \ref{wk} that
\begin{equation}
  \lambda_k^*(\CR)^{d/2} = \sum_{j \in J} n_j \lambda_{j}(\Omega)^{d/2}
\end{equation}
where $\set{n_j: j \in J}$ are nonnegative integers such that
\begin{equation}
 \sum_{j \in J} n_j j = k.
\end{equation}
Therefore
\begin{equation}
\begin{aligned}
 \frac{\lambda_k^*(\CR)^{d/2}}{k} &= \frac 1 k \sum_{j \in J} n_j
 \lambda_{j}^{d/2} \, \ge \, \frac 1 k \sum_{j \in J} n_j j L' = L' ,
 \end{aligned}
\end{equation}
which immediately implies 
\begin{equation}
 L = \inf_{k} \frac{\lambda_k^*(\CR)^{d/2}}{k} \ge L' = \min_{j \in J} \frac{\lambda_{j}(\Omega)^{d/2}}{j} \ge L.
\end{equation}

Furthermore, since $\lambda_k(\Omega) \ge \lambda_k^*(\CR)$ for every $k \in \mathbb{N}$ and since Weyl's law implies that
\begin{equation}
 \lim_{k \to \infty} \frac{\lambda_k(\Omega)^{d/2}}{k} = \frac{(2\pi)^d}{\omega_d},
\end{equation}
we get from Corollary \ref{BLY} that indeed
\begin{equation}
 \min_{j \in J} \frac{\lambda_{j}(\Omega)^{d/2}}{j} = \lim_{k \to \infty} \frac{\lambda_k^*(\CR)^{d/2}}{k} \le \frac{(2\pi)^d}{\omega_d}
\end{equation}

\end{proof}\vspace{6pt}

\begin{proof}[Proof of Theorem \ref{thm:trichotomy}]
  We have proved in Proposition \ref{Infinite} that if $J$ is infinite, then
  P\'olya's conjecture holds. The two other parts of the trichotomy are proved by Proposition \ref{Minimum}.
  
\end{proof}
We now turn our attention to the proof of Theorem \ref{thm:tfae}, in the case
where the domain $\Omega$ satisfies the two-term Weyl law
\eqref{twotermweyldir}.

\begin{proof}[Proof of Theorem \ref{thm:tfae}]
  In all generality, clearly (2) implies (3), and (1) implies (3). Indeed, since (1) places in the first
  possibility of the trichotomy \ref{thm:trichotomy}, which implies (3). We
  shall show that the assumption that a two-term Weyl law holds can be used to
  infer that (1) implies (2) and that
  (3) implies (1).

  \noindent  \textbf{Proof of (1) implies (2).} 
  Write the sequence of minimisers, all of volume $1$, as
  \begin{equation}
    \Omega_k^{*} = \bigsqcup_{q=1}^{\nu_k} r_{k,q} \Omega,
  \end{equation}
  where $\nu_k := \nu(\Omega_k^*) < \infty$ by Lemma \ref{existence_min}. Suppose that the $r_{k,q}$ coefficients are
  in decreasing order,
  $$r_{k,1} \ge \dotso \ge r_{k,\nu_k}.$$
  It follows from Lemma \ref{wk} that for every $1 \le q
  \le \nu_k$ there is $j_q:= j_q(k) \in J$ such that 
  \begin{equation}
    r_{k,q} = \left( \frac{\lambda_{j_q}(\Omega)}{\lambda_k^*(\CR)}
    \right)^{1/2},
  \end{equation}
  and $j_1 + \dotso + j_{\nu_k} = k$. It follows from Weyl's law that
  \begin{equation}
    \lim_{k\to \infty} r_{k,1} = 1 \quad \Longleftrightarrow \quad    \lim_{k \to \infty} \frac{j_1(k)}{k} = 1.
  \end{equation}
  Suppose that the righthand side of the previous equivalence does not hold,
  i.e. that there exists $\delta > 0$ and a
  subsequence, that we still label with $k$, such that for all $k$, $j_1(k) \le (1
  - \delta)  k$. 
 For all $\eps > 0$, it follows from the two-term Weyl law that there
exists a rank $N$ such that for all $j > N$,
\begin{equation} \label{eq:bound2eps}
  \lambda_j(\Omega)^{d/2} \ge \frac{(2\pi)^d}{\wv d}j + 
  \underbrace{\left(\frac{(2\pi)\wv{d-1}}{4
  \wv d^{\frac{2d-1}{d}}}\abs{\del \Omega}- \eps \right)}_{:=A - \eps}j^{\frac{d-1}{d}}.
\end{equation}
For all $k$, let $Q:= Q(k)$ be defined as
\begin{equation}
  Q := \begin{cases}
    0 &\text{if } j_q \le N \text{ for all } 1 \le q \le \nu_k, \\
    \max\set{q : j_q > N} & \text{otherwise.}
  \end{cases}
\end{equation}
We define
\begin{equation}
  \Upsilon_k := \bigsqcup_{q=1}^{Q} r_{k,q} \Omega \quad \text{and} \quad
  \Xi_k := \bigsqcup_{q=Q+1}^{\nu_k} r_{k,q} \Omega.
\end{equation}
We claim that $\nu(\Xi_k)$ is bounded in $k$. Indeed, it follows from the strong
P\'olya conjecture that there exists $M$ such that for all $j > M$, 
\begin{equation}
  \frac{\lambda_j(\Omega)^{d/2}}{j} < \frac{\lambda_{j_q}(\Omega)^{d/2}}{j_q}
\end{equation}
for all $q > Q$. Writing $$j_{Q +1} +
\dotso + j_{\nu_k} = j' > \nu(\Xi_k),$$
it follows from Lemma \ref{wk} that if $j' \ge M$, then
\begin{align}
  \lambda_{j'}(\Xi_k)^{d/2}
   &=
  \lambda_{j'}^*(\CR)^{d/2} \\
  &\le \sum_{q=Q + 1}^{\nu_k} j_q \frac{\lambda_{j'}(\Omega)^{d/2}}{j'} \\
  &< \sum_{q=Q+1}^{\nu_k} \lambda_{j_q}(\Omega)^{d/2} \\
  &= \lambda_{j'}(\Xi_k)^{d/2},
\end{align}
a contradiction. Hence, $\nu(\Xi_k) \le j' < M$, and
\begin{equation}
  \label{eq:k0}
k_0 := \sum_{q =1}^{Q} j_q  = \ge k - M
\end{equation}
Recall that we assumed that there is $\delta > 0$ such that $j_1 < (1 - \delta)
k$, and it follows from \eqref{eq:k0} that, up to choosing $\delta$ a bit
smaller, $j_1 < (1 - \delta) k_0$. Let $R := R(k)$ be defined as
$$
R:= \max\set{r : 2 \le r \le Q \text{ and } \frac{1}{k_0} \sum_{q=r}^Q j_q >
\delta},
$$
and we denote
$$
\delta_k := \frac{1}{k_0}\sum_{q = R}^Q j_q.
$$
There is no loss of generality in assuming $\delta < 1/3$. That the $j_q$ are in
decreasing order ensures that in that case
$\delta < \delta_k < 1 - \delta$. 
Recall that for all $1 \le q \le Q$, $j_q > N$ hence \eqref{eq:bound2eps}
holds. It is a consequence again of Lemma \ref{wk} that
\begin{equation}
\begin{aligned}
  \label{eq:tocontradict}
  \lambda_k^*(\CR)^{d/2} &\ge \sum_{q=1}^{Q}
  \lambda_{j_q}(\Omega)^{d/2} \\
  & \ge \sum_{q = 1}^{Q} \left[\frac{(2\pi)^d}{\wv d} j_q + 
  \left(A - \eps
\right)j_q^{\frac{d-1}{d}}\right] \\
&\ge \frac{(2\pi)^d}{\wv d} k +  
 \left( A - \eps
\right)\sum_{q = 1}^{Q}j_q^{\frac{d-1}{d}} + \bigo{1}.
\end{aligned}
\end{equation}
We study the sum in the last line of the previous display. It follows
from subadditivity of the function $x \mapsto x^\alpha$ for $\alpha < 1$, and
from $k_0^\alpha = k^\alpha + \bigo{k^{\alpha - 1}}$ that
\begin{align}
\sum_{q = 1}^{Q} j_q^{\frac{d-1}{d}} &\ge \left(\sum_{q = 1}^{R - 1} j_q\right)^{\frac{d-1}{d}} +
  \left(\sum_{q=R}^{Q} j_q\right)^{\frac{d-1}{d}} \\
&\ge \left((1-\delta_k)^{\frac{d-1}{d}} + \delta_k^{\frac{d-1}{d}}\right)
k^{\frac{d-1}{d}} + \bigo{k^{-1/d}}
\end{align}
It is a simple exercise to see that the
function $x \mapsto x^\alpha + (1-x)^\alpha$, $\alpha < 1$ being concave and
symmetric on
$[0,1]$ and $\delta < \delta_k < 1 - \delta$ imply that
\begin{equation}
  (1 - \delta_k)^{\frac{d-1}{d}} + \delta_k^{\frac{d-1}{d}} \ge (1 -
  \delta)^{\frac{d-1}{d}} + \delta^{\frac{d-1}{d}} \ge 1 +
\left(2^{1/d} - 1\right)\delta =: 1 + c_d \delta,
\end{equation}
and $c_d > 0$. 
Putting this back into \eqref{eq:tocontradict}, it follows that
\begin{equation}
  \lambda_{k}^*(\CR)^{d/2} \ge \frac{(2\pi)^d}{\wv d} k + 
 \left(A - \eps\right)(1 + c_d \delta)
 k^{\frac{d-1}{d}} + \bigo 1.
\end{equation}
Choosing 
\begin{equation}
  \eps = \frac{c_d \delta A}{2(1 + c_d \delta)}
\end{equation}
gives, for $k$ large enough, that
\begin{equation}
  \lambda_k^*(\CR)^{d/2} \ge \frac{(2\pi)^d}{\wv d}k + \left(A + \frac{A c_d
  \delta}{3}\right)k^{\frac{d-1}{d}}.
\end{equation}
However, since
\begin{equation}
  \lambda_k(\Omega)^{d/2} = \frac{(2\pi)d}{\wv d} k + A k^{\frac{d-1}{d}} +
  \smallo{k^{\frac{d-1}{d}}},
\end{equation}
we have that for $k$ large enough, $\lambda_k(\Omega)^{d/2} <
\lambda_k^*(\CR)^{d/2}$, a contradiction. Hence, for any $\delta > 0$ there are no subsequences along
which $j_1(k) < (1 - \delta)k$ for all $k$. It is readily seen that $r_{k,1}$ converges to
$1$. 

\noindent \textbf{Proof of (3) implies (1).} Assume that the Strong P\'olya
conjecture doesn't hold for $\Omega$. From Theorem \ref{thm:trichotomy}, there
is a rank $j$ such that 
\begin{equation}
  \label{eq:nostrongpol}
  \lambda_j(\Omega)^{d/2} \le \frac{(2\pi)^d}{\wv{d}} j. 
\end{equation} 
Suppose that along a subsequence, labeled by $k$,
\begin{equation}
  \Omega_k^* = (1 - \eps_k) \Omega \sqcup \Upsilon_k,
\end{equation}
with $\eps_k \to 0$. From Lemma \ref{wk}, for every $k$ there exists a rank
$j_k$ such that
\begin{equation}
  \label{eq:jkunbdd}
  \lambda_k^*(\CR)^{d/2} = (1 - \eps_k)^{-d} \lambda_{j_k}(\Omega)^{d/2},
\end{equation}
and that $\Omega = \Omega_{j_k}^*$. It follows from Lemma \ref{BLY} and equation
\eqref{eq:jkunbdd} that $j_k \to \infty$. By the two-term Weyl law
\eqref{twotermweyldir}, there is $A > 0$ such that
\begin{equation}
  \lambda_{j_k}(\Omega)^{d/2} = \frac{(2 \pi)^d}{\wv{d}} j_k + A
  j_k^{\frac{d-1}{d}} + \smallo{j_k^{\frac{d-1}{d}}},
\end{equation}
hence there exists a constant $C > 0$ such that for every $k$ large enough,
\begin{equation}
  \label{eq:wellsep}
  \lambda_{j_k}(\Omega)^{d/2} - \frac{(2\pi)^d}{\wv{d}} j_k \ge C
  j_k^{\frac{d-1}{d}}.
\end{equation}
We now show that for large enough $k$, $\Omega$ is in fact not a minimiser for
$\lambda_{j_k}$ amongst $\CR$. Write $j_k = n_k j + r$, with $0 \le r < j$.
Consider the domain $\Omega'$ defined as
\begin{equation}
  \label{eq:omega'}
  \Omega' = \left(\frac{\lambda_r(\Omega)}{\lambda_j(\Omega)}\right)^{1/2}
  n_k^{-1/d}
  \Omega \sqcup \left(\bigsqcup_{q =1}^{n_k} n_k^{-1/d} \Omega\right).
\end{equation}
We have constructed $\Omega'$ explicitly so that the first component in
\eqref{eq:omega'} has $n_k^{2/d}\lambda_j(\Omega)$ as its $r^{\rm{th}}$
eigenvalue, and all the other components have $n_k^{2/d}\lambda_j(\Omega)$ as
its $j^{\rm{th}}$ eigenvalue, it then follows that
\begin{equation}
  \lambda_{j_k}(\Omega')^{d/2} = n_k \lambda_j(\Omega)^{d/2}.
\end{equation}
Furthermore,
\begin{equation}
    \abs{\Omega'} = \left( 1+ \left(
      \frac{\lambda_r(\Omega)}{\lambda_j(\Omega)} \right)^{d/2} \frac{1}{n_k}
    \right).
\end{equation}
Combining these equalities with \eqref{eq:nostrongpol}, we deduce that
\begin{equation}
  \begin{aligned}
    \abs{\Omega'} \lambda_{j_k}(\Omega')^{d/2} 
     &\le \left( 1+ \left(
     \frac{\lambda_r(\Omega)}{\lambda_j(\Omega)} \right)^{d/2} \frac{1}{n_k}
    \right)\frac{(2\pi)^d}{\wv{d}} n_k j \\
     &= \left( 1+ \left(
       \frac{\lambda_r(\Omega)}{\lambda_j(\Omega)} \right)^{d/2}
       \frac{j}{j_k - r}
     \right)\frac{(2\pi)^d}{\wv{d}}(j_k - r) \\
     &= \frac{(2\pi)^d}{\wv{d}} j_k + \bigo{1}
\end{aligned}
\end{equation}
This combined with estimate \eqref{eq:wellsep} implies that for $k$ large
enough, $\abs{\Omega'}\lambda_{j_k}(\Omega')^{d/2} <
\lambda_{j_k}(\Omega)^{d/2}$, contradicting optimality of $\Omega$ for
$\lambda_{j_k}$.
\end{proof}
For the next few results we shall assume that $\Omega$ is a minimiser only finitely many times, namely $\Omega = \Omega_k^*$ if and only if
$k \in J = \{j_1, \dots, j_p\} \subset \N$. We shall continue to write simply $L = \inf_k \, k^{-1} \lambda_k^*(\CR^{d/2})$. We shall say that a
\textit{minimiser $\Omega_k^*$ realises $L$} if $ \lambda_k(\Omega_k^*)^{d/2} \, k^{-1}= L$. 

Recall that a for set $\Upsilon \in \CR$ and $n \in N$, the
$n$-th propagation of $\Upsilon$ is the set
\begin{equation}
 \Upsilon^{(n)} = \bigsqcup_{\ell = 1}^n \, \frac{1}{n^{1/d}} \,  \Upsilon \, .
\end{equation}
Observe that $|\Upsilon| = |\Upsilon^{(n)}|$ and that $\lambda_k(\Upsilon)^{d/2} \, k^{-1} = \lambda_{nk}(\Upsilon^{(n)})^{d/2} \, (nk)^{-1}$ for any $n \in \N$.

\begin{defi}\label{DefPropagation}
A minimiser $\Omega_k^*$ \textit{propagates as a minimiser in $\CR$} if for every $n \in \N$ we have $\Omega_k^{* \, (n)} = \Omega_{nk}^*$. A minimiser $\Omega_k^*$
\textit{weakly propagates as a minimiser in $\CR$} if there exist a sequence of integers $n_1 < n_2 < \dots \nearrow + \infty$ and a corresponding sequence of minimisers
in $\CR$ of the form
\begin{equation} \label{weakpropa}
 \Omega_{k'_i}^* = r_i \Omega_k^{* \, (n_i)} \, \sqcup \, \Upsilon_i \,. 
\end{equation}
\end{defi} \vspace{6pt}

\begin{prop} \label{PropPropagation}
A minimiser $\Omega_k^*$ realises $L$ if and only if it propagates as a minimiser in $\CR$.
\end{prop}

\begin{proof}
Fix $k \in \N$ and a minimiser $\Omega_k^*$. We have $\lambda_{nk}^*(\CR) \le \lambda_{nk}\left(\Omega_{k}^{(n)}\right) = n^{2/d} \lambda_k(\Omega_k^*)$. Fix $n >1$. Whether or not $nk$ belongs to $J$, there exist nonnegative integers $n_1, \dots, n_p$ such that $nk = \sum_{i=1}^p n_i j_i$ and
\begin{equation}
 \begin{aligned}
 \lambda_{nk}^*(\CR)^{d/2} &= \sum_{i = 1}^p n_i \lambda_{j_i}^{d/2} \, \ge \,  \sum_{i = 1}^p n_i j_i L = nk L.
 \end{aligned}
\end{equation}
Therefore we have
\begin{equation} \label{doubleineq}
 L  \le \frac{\lambda_{nk}^*(\CR)^{d/2}}{nk} \le \frac{\lambda_{nk}\left(\Omega_k^{(n)}\right)^{d/2}}{nk} = \frac{\lambda_{k}(\Omega_k^*)^{d/2}}{k} \, .
\end{equation}

In view of this, it follows that $\Omega_k^*$ realises $L$ if and only if for every $n \in \N$ both inequalities in \eqref{doubleineq} are equalities.

In turn, this is equivalent to only the second inequality being an equality for every $n$. Indeed, the latter would imply that the sequence $n \mapsto \lambda_{nk}^*(\CR)^{d/2} \, (nk)^{-1}$ is constant, but we know that it converges to $L$ as $n \to \infty$ hence the first inequality being an equality too.

Now for any fixed $n$, the equality $\lambda_{nk}^*(\CR)^{d/2} \, ( nk)^{-1} = \lambda_{nk}\left( \Omega_k^{(n)}\right)^{d/2} \, (nk)^{-1}$ is equivalent to the claim that $\Omega_k^{(n)}$ realises $\lambda_{nk}^*(\CR)$. Consequently, the second inequality in \eqref{doubleineq} being an equality for every $n \in N$ means precisely that $\Omega_k^*$ propagates as a minimiser.
\end{proof} \vspace{6pt}

\begin{lem}\label{propag=weakpropag}
A minimiser $\Omega_k^*$ propagates as a minimiser in $\CR$ if and only if it weakly propagates as a minimiser in $\CR$.
\end{lem}

\begin{proof}
The "only if" part is trivial. For the "if" part, consider a sequence of minimisers $\Omega_{k_j}^*$ as in equation \eqref{weakpropa}. It follows from Lemma \ref{wk} that for each $j \in \N$, the set $\Omega_k^{* \, (n_j)}$ realises $\lambda_{n_j k}^*(\CR)$. As a consequence of this and of Corollary \ref{BLY}, we compute
\begin{equation}
\begin{aligned}
 \frac{\lambda_k(\Omega_k^*)^{d/2}}{k} = \frac{\lambda_k \left(\Omega_k^{* \, (n_j)} \right)^{d/2}}{n_jk} = \frac{\lambda_{n_jk}^* (\CR)^{d/2}}{n_jk} \underset{j \to + \infty}{\longrightarrow} L \, .
\end{aligned}
\end{equation}
This means that $\Omega_k^*$ realises $L$. Proposition \ref{PropPropagation} thus implies that $\Omega_k^*$ propagates as a minimiser in $\CR$.
\end{proof} \vspace{6pt}

Let us consider the sets
\begin{equation}
K_{L} = \{ \, k \in \N \, : \, \lambda_k^*(\CR)^{d/2} \, k^{-1} = L \,  \} \; \mbox{ and } \; J_{L} = J \cap K_{L} \, .
\end{equation}
We observe that $K_{L}$ is closed under finite sums.

Continuing with the assumption of finite $J$, Proposition \ref{Minimum} implies that the set is not empty. Set $j_{L} = \mathrm{max} \, J_{L}$. Proposition \ref{PropPropagation} implies that the minimiser $\Omega = \Omega_j^*$ associated to $j \in J_{L}$ propagates as a minimiser in $\CR$. One might expect these minimisers to be special, for instance to have a minimum numbers of connected components among minimisers of a given eigenvalue functional, if not to be unique. These expectations are even more vivid for the propagations of $\Omega^*_{j_{L}}$. The next result investigates these possibilities. \vspace{6pt}

\begin{lem}\label{optimality} Assume $J_{L}$ is finite and set $j_L = :=\max
  J_L$. Let $j \in J_L$. If there exist $n \in \N$ and a minimiser $\Omega^*_{n
  j} \neq \Omega^{* \, (n)}_{j}$, then $\{j\} \subsetneq J_{L}$. If furthermore
  $\nu \left( \Omega^*_{n j}  \right) \le \nu \left( \Omega^{* \, (n)}_{j}
  \right)$, then $j < j_{L}$. If instead $\nu \left( \Omega^*_{n j}  \right) >
  \nu \left( \Omega^{* \, (n)}_{j}  \right)$, then there exists $j' \in J_L$
  such that $j' < j$.
\end{lem}

\begin{proof} Both $\Omega_j^{* \, (n)}$ and $\Omega^*_{nj}$ realises $\lambda_{nj}^*(\CR)$. As a result of Lemma \ref{wk} we have a decomposition
\begin{equation}\label{decompo}
\Omega^*_{nj} = \bigsqcup_{i=1}^p \, \bigsqcup_{m=1}^{n_i} \, r_i \Omega^*_{j_i} \; \mbox{ with } \; \sum_{i=1}^p n_i j_i = nj
\end{equation}
which induces the equality
\begin{equation}
\lambda^*_{nj}(\CR)^{d/2} = \sum_{i=1}^p \, n_i \lambda^*_{j_i}(\CR)^{d/2} \, .
\end{equation}

We claim that there is an index $h$ such that $j_h \neq j$ and $n_h > 0$. Otherwise the only positive $n_i$ would be $n_l$ where $j_l = j$; It would follow from \eqref{decompo} that $n_l = n$ and that $r_i = n_l^{-1/d}$, hence $\Omega_{nj}^* = \Omega_j^{(n)}$. This is a contradiction with our assumptions, hence the claim.

Since $j \in J_{L}$, $\Omega_j^{* \, (n)}$ realises $L$ and so does $\Omega^*_{nj}$. We compute
\begin{equation}
\begin{aligned}
L &= \frac{\lambda_{nj}(\Omega_{nj}^*)^{d/2}}{nj} = \frac{1}{nj} \sum_{i=1}^p \, n_i \lambda^*_{j_i}(\CR)^{d/2} \\
&= \frac{1}{nj} \sum_{i=1}^p \, n_i j_i \frac{\lambda^*_{j_i}(\CR)^{d/2}}{j_i} \ge \frac{1}{nj} \sum_{i=1}^p \, n_i j_i L = L \, ,
\end{aligned}
\end{equation}
which implies that $\lambda_{j_i}^*(\CR)^{d/2} j_i^{-1} = L$ for every $i$ such that $n_i > 0$, so in particular for $i=h$. This means  $j_h \neq j$ satisfies $j_h \in J_{L}$, hence $\{j \} \subsetneq J_{L}$.

Assume now moreover $\nu \left( \Omega^*_{n j}  \right) \le \nu \left( \Omega^{* \, (n)}_{j}  \right) = n$. By the pigeonhole principle and -- in case the previous inequality is an equality -- by $\Omega^*_{n j} \neq \Omega^{* \, (n)}_{j}$, at least one of the connected components of $\Omega^{* \, (n)}_{j}$ has volume strictly greater than $n^{-1}$. Put differently, if $h$ is the index of such a component then $r_h > n^{-1/d}$. We compute
\begin{equation}
\begin{aligned}
L &= \frac{\lambda_{j_h}^*(\CR)^{d/2}}{j_h} = \frac{r_h^d \,  \lambda_{j_h}(r_h \Omega_{j_h}^*)^{d/2}}{j_h} \\
&= \frac{r_h^d \, \lambda_{nj}^*(\CR)^{d/2}}{j_h} \, > \, \frac{n^{-1}
\lambda_{nj}^*(\CR)^{d/2}}{j_h} = \frac{j}{j_h} \,
\frac{\lambda_{nj}^*(\CR)^{d/2}}{nj} = \frac{j}{j_h} \, L \, , 
\end{aligned}
\end{equation}
which means that $j < j_h$ and \textit{a fortiori} that $ j < j_{L}$.

Assume now instead $\nu \left( \Omega^*_{n j}  \right) > \nu \left( \Omega^{* \, (n)}_{j}  \right) = n$. The pigeonhole principle now implies that at least one connected component has volume strictly less than $n^{-1}$. The same argument as before with the direction of inequalities inverted yields the existence of $j' = j_h \in J_{L}$ such that $j > j'$.
\end{proof} \vspace{6pt}

A consequence of this last lemma is that for $n \in \N$, the domain with the least number of connected components realising the eigenvalue $\lambda_{nj_{L}}^*(\CR)$ is unique and is given by the propagation $\Omega_{j_{L}}^{* \, (n) }$.

Another consequence of the proof is that $K_{L}$ is generated by $J_{L}$, that is any $k \in K_{L}$ is a finite sum of elements in $J_{L}$. Indeed, given $k \in K_{L} \setminus J_{L}$ and a minimiser $\Omega_k^*$, the propagation $\Omega_k^{* \, (j_{L})}$ realises $\lambda_{j_{L}k}^*(\CR)$. The connected components of this propagation are thus contracted copies of minimisers canonically associated with $J_{L}$ and so are the ones of $\Omega_k^*$, hence the result. The minimisers $\Omega_{j_i}^* = \Omega$ with $j_i \in J_{L}$ are thus the building blocks of any minimiser realising $L$.



\section{Bounds from packings}\label{sec:packing}

  We have just seen that the failure of P\'olya's conjecture
  for a domain $\Omega$ implies that infinitely many minimisers in $\CR(\Omega)$
  are realised by propagators $\Omega^{(n)} = \cup_{j=1}^n n^{-1/d} \Omega$. It is
  thus natural to study the spectrum of those propagators, notably by
  geometrically realising them as subsets of other domains, that is by packing
  the $\Omega^{(n)}$s into others domains. This packing idea leads to the main
  result in this section, to wit an estimate from below on $L = \inf_{k \in \N}
  \,  \, \lambda^*_k(\CR)^{d/2} \, k^{-1}$ in term of the "packing density" of
  $\Omega$. Recall that this packing density was defined in Definition \ref{packing} \vspace{6pt}

  We start by proving a few properties of this packing density.

\begin{lem}
Given three bounded domains $\Omega$, $V$ and $W$,
\[ \rho_{\Omega, W} \, \ge \, \rho_{\Omega, V} \, \rho_{V, W} \, . \]
\end{lem}

\begin{proof}
Given any $\eps > 0$, there exist a packing $g$ of $\Omega^{(m)}$ into $V$ of density $\rho_g > \rho_{\Omega, V} - \eps$ and an asymptotic packing $P = \{(n_i, \rho_i, f_i)\}_{i \in \N}$ of $V$ into $W$ with asymptotic density $\rho_P > \rho_{V, W} - \eps$. It is very clear how $g$ and $P$ can be "composed" to yield an asymptotic packing of $\Omega$ into $W$ with asymptotic density $\rho_g \rho_P > \rho_{\Omega, V}\rho_{V, W} - O(\eps)$. The lemma readily follows. 

\end{proof} \vspace{6pt}

\begin{prop} \label{tiling}
Let $\Omega$ and $V$ be two bounded domains in $\R^d$ with volume $1$. Suppose that $\Omega$ tiles $\R^d$ and that the upper Minkowski dimension of $\partial V$ is strictly smaller than $d$. Then $\rho_{\Omega, V}=1$ and thus $\rho_{\Omega} = 1$.
\end{prop} \vspace{6pt}

\begin{rem} We recall that the upperbox dimension or upper Minkowski dimension of a set $S \subset \R^d$ could be defined as
\[ d_{\mbox{\small{up}}}(S) := d - \liminf_{r \to 0^+} \, \frac{\log | S(r)|}{\log r}  \]
where $S(r) := \{ y \in \R^d \, : \, \| y - S \| < r \}$ is the $r$-neighborhood of $S$.
\end{rem}\vspace{6pt}

\begin{proof}
For simplicity, suppose $0 \in \mathrm{int}(V) \subset \R^d$ and consider that any homothety to be performed below is with respect to $0$. We shall also think of the tiling $F$ as a mere quasi-inclusion and we will not use $F$ in our notations.

Since $V$ is bounded, there exists $R > 0$ such that $V \subset rV$ for all $r
\ge R$. Consequently, we have the sequence of inclusions $$V \subset RV \subset
R^2V \subset R^3V \subset \dots$$ Without lost of generality, take $R \in \N$.

Denote $\Omega_i$ the $i$-th component $\Omega$ in the disjoint union $\sqcup_{i
\in \N} \, \Omega$. For $n \in \N$, let $I_n \subset \N$ be the largest set such
that $\overline{\Omega_i} \subset n V$ for every $i \in I_n$. This set is finite
as its cardinality is at most $|nV|/|\overline{\Omega}| = n$. Because of the
previous paragraph, $I_{R^i} \subset I_{R^{i+1}}$ for every $i \in \N$. For $i
\in \N$, set $n_i = \# \, I_{R^{i}}$.

Because $\Omega$ and hence $\overline{\Omega}$ are bounded, the latter is contained in an open ball $B$ of diameter $D$. Let 
\begin{equation}
 (nV)_{2D} = \{ \, p \in nV \, : \, \mathrm{dist}(p, (nV)^{c}) \ge 2D \, \} \, .
 \end{equation}
We claim that the set $(nV)_{2D} \setminus \cup_{i \in I_n} \,
\overline{\Omega_i}$ is empty. Suppose otherwise; then there exist a point $x$ in this nonempty set and, since $\Omega$ is a tile, an index $i \in (I_n)^{c}$ such that $x \in \overline{\Omega_i} \subset B_i$. The definition of $I_n$ implies $\overline{\Omega_i} \cap (nV)^{c} \neq \emptyset$, so there exists $y$ in this latter intersection and thus in $B_i$. It follows that $\mathrm{dist}(x,y) < 2D$, which is a contradiction. This proves the claim, and consequently $(nV)_{2D} \subset \cup_{i \in I_n} \overline{\Omega_i} \subset nV$.

By assumption on $\partial(nV)$, the volume of the $2D$-neighbourhood of $\partial(nV)$ grows like $o(n^{d})$, so that the volume of $(nV)_{2D}$ grows like $n^d - o(n^d)$. From the set inclusions obtained in the previous paragraph, the same asymptotic is true for the growth of the volume of $\cup_{i \in I_n} \, \overline{\Omega_i}$, that is of $\sharp \, I_n$.

Consider the asymptotic packing $P = \{(n_i, \rho_i, f_i)\}_{i \in \N}$ given by $n_i = \sharp \, I_{R^{i}}$, $\rho_i = n_i/n^d$ and
\begin{equation}
 f_i \, : \, \overline{\Omega^{(n_i)}} \cong \sqcup_{i \in I_{R^{i}}} \, n_i^{-1/d} \, \overline{\Omega_i} \; \hookrightarrow \;  n_i^{-1/d} \, nV = \rho_i^{-1/d} \, V \; . 
 \end{equation}
From the previous paragraph we get $\rho_P = \lim_{i \to \infty} \rho_i = 1$, thus $\rho_{\Omega, V} =1$.
\end{proof} \vspace{6pt}

The previous result suggests to define another, \emph{a priori} smaller notion of packing density, namely the \emph{lower packing number or lower packing density of $\Omega$} is
\begin{equation} 
\underline{\rho}_{\Omega} \, = \, \inf \, \left\{ \, \rho_{\Omega, V} \, \left| \,\mbox{$V$ bounded domain, $|V| = 1$, $d_{\mbox{\small{upperbox}}}(\partial V) < d$ } \right. \, \right\} \; . 
\end{equation}\vspace{0pt}

\begin{cor}
For any bounded domain $\Omega \subset \R^d$, 
\[ \underline{\rho}_{\Omega} = \rho_{\Omega, V} > 0 \]
for any bounded tile $V \subset \R^d$ whose boundary has upperbox dimension strictly less than $d$.
\end{cor}

\begin{proof}
Let $W \subset \R^d$ be any bounded domain whose boundary has upper Minkowski dimension strictly less than $d$. Then from the two previous results we get $\rho_{\Omega, W} \ge \rho_{\Omega, V} \rho_{V, W} = \rho_{\Omega, V}$. Taking the infimum over all $W$ yields the equality claimed in the statement.

To prove the inequality, let's take $V = [0,1]^d$. Since $\Omega$ is bounded, there clearly is some $\rho \in (0,1]$ such that $\Omega$ can be packed in $\rho^{-1/d}V$. Since $V^{(i^d)}$ fully pack $V$ for each integer $i$, by "composing" packings we deduce that there is at least one asymptotic packing of $\Omega$ into $V$ with constant density $\rho > 0$, and \emph{a fortiori} we get $\rho_{\Omega, V} > 0$.
\end{proof}\vspace{6pt}

\begin{rem}
In the few last results, the assumption on the upperbox dimension -- which guaranteed that the boundary had vanishing Lebesgue measure -- was not superfluous. Indeed, given any $\eps > 0$, it is possible to find a bounded tile $V_{\eps} \subset \R^d$ with volume $1$ such that $| \mathrm{int}( V_{\eps})| < \eps$, for instance by applying a suitable symmetric adaptation of Knopp's construction of a Osgood "surface" on the sides of a cube; the packing density of a typical domain $\Omega$ into $V_{\eps}$ would thus be smaller than $\eps$. We leave the details to the industrious reader.
\end{rem}\vspace{6pt}

We are now in a position to prove a lower bound on $L := \mathrm{inf}_k \, \lambda^*_k(\CR)^{d/2} \, k^{-1}$ for the Dirichlet Laplacian eigenvalue problem in the class $\CR(\Omega)$.

\begin{lem} \label{lowerpacking}
Assume that $K_{L} = \{ \, k \in \N \, :  \,  \lambda^*_k(\CR)^{d/2} \, k^{-1} = L \, \}$ is non-empty. Then
\begin{equation}\label{eqlowerbound}
L \ge \rho_{\Omega} \, \frac{(2 \pi)^d}{\omega_d} \; .
\end{equation}\vspace{0pt}
\end{lem}

\begin{proof}
Using Lemma \ref{wk} in a way we already repeatedly used it before, we deduce from the assumption $K_{L} \neq \emptyset$ that $J_{L} = \{ \, k \in K_{L} \, : \, \Omega = \Omega^*_k    \,   \} \neq \emptyset$. Pick some $j \in J_{L}$.

Let $\eps > 0$ and consider an open bounded domain $V$ with volume $1$ such that $\rho_{\Omega, V} \ge \rho_{\Omega} - \eps/2$. Consequently, there exist an asymptotic packing $P = \{(n_i, \rho_i, f_i) \}_{i \in \N}$ of $\Omega$ into $V$ such that $\rho_P = \lim_{i \to \infty}  \rho_i \ge \rho_{\Omega} - \eps$.

The isometric quasi-embedding $f_i : \Omega^{* \, (n_i)}_j \to \rho_i^{-1/d} \, V$ allows us to view $\Omega^{* \, (n_i)}_j$ as a genuine subset of $\rho_i^{-1/d} \, V$. Considering the well-known fact that any Dirichlet eigenvalue functional $\Upsilon \mapsto \lambda_k(\Upsilon)$ is decreasing with respect to inclusion, namely that $\Upsilon_1 \subset \Upsilon_2$ implies $\lambda_k(\Upsilon_1) \ge \lambda_k(\Upsilon_2)$, it follows that 
\begin{equation}
\lambda_{n_i j} \left( \Omega^{* \, (n_i)}_j  \right)^{d/2} \ge \lambda_{n_i j} \left( \rho_i^{-1/d} \, V \right)^{d/2} \; .
\end{equation}
The left-hand side is equal to $n_i \lambda_j(\Omega^*_j) = n_i j L$, whereas the right-hand side equals $\rho_i \, \lambda_{n_i j}(V)^{d/2}$. Therefore
\begin{equation}
L \, \ge \,  \rho_i \; \frac{\lambda_{n_i j} \left(  V \right)^{d/2}}{n_i j} \; .
\end{equation}
Since $\lim_{i \to \infty} \rho_i = \rho_P \ge \rho_{\Omega} - \eps$ and because of Weyl's asymptotic law, taking the limit $i \to + \infty$ on the right-hand side yields
\begin{equation}
L \, \ge \,  (\rho_{\Omega} - \eps) \; \frac{(2 \pi)^d}{\omega_d} \; .
\end{equation}
As this is true for any $\eps > 0$, the result follows.
\end{proof}\vspace{6pt}
Theorem \ref{thm:packing} follows as a corollary of the previous Lemma.

\begin{proof}[Proof of Theorem \ref{thm:packing}]
The set $J = \{ \, k \in \N \, : \, \Omega \mbox{ realises } \lambda^*_k(\CR) \, \}$ is either infinite or finite. If it is infinite,
Proposition \ref{Infinite} implies P\'olya's conjecture and \textit{a fortiori} \eqref{eqlowerbound} as $\rho_{\Omega} \le 1$. If instead it is finite,
then Proposition \ref{Minimum} implies that $K_{L}$ is non-empty and the claim follows from the previous lemma.
\end{proof} \vspace{6pt}

We finally prove that simple tiles satisfy the strong P\'olya conjecture under
the condition that the two-term Weyl law holds, for Dirichlet eigenvalues.

\begin{proof}[Proof of Theorem \ref{polya_stile} for Dirichlet eigenvalues]
Fix a rank $j$ for which $\Omega$ realises $\lambda_j^*$. Since $\Omega$ is a
simple tile, there are a domain $V \subset \R^d$ with volume $1$ and an
asymptotic packing $P = \{(n_i, 1, f_i)\}_{i \in \N}$ of $\Omega$ into $V$ with
constant packing density $1$. Since $\Omega$ satisfies the two-term Weyl law
\eqref{twotermweyldir}, there is $M \in \N$ such that
\[ \frac{\lambda_{m}(V)^{d/2}}{m} > \frac{(2\pi)^d}{\omega_d} \quad \forall \, m
\ge M \, . \]
Consider $i \in \N$ sufficiently large so that $n_i j \ge M$, and consider the (full) packing $f_i : \Omega^{(n_i)} \to V$. Invoking the monotonicity of Dirichlet eigenvalues we thus get
\[ \frac{\lambda_{j}(\Omega)^{d/2}}{j} = \frac{\lambda_{n_i j}(\Omega^{(n_i)})^{d/2}}{n_i j} \ge \frac{\lambda_{n_i j}(V)^{d/2}}{n_i j} > \frac{(2\pi)^d}{\omega_d} \, .   \]
Considering Proposition \ref{Minimum}, this implies that $\Omega$ realises
$\lambda_k^*$ infinitely often. Using Corollary \ref{BLY}, we deduce
\[ L := \inf_{k} \frac{\lambda_k^*(\CR)^{d/2}}{k} = \inf_{j \in J} \frac{\lambda_j^*(\CR)^{d/2}}{j} .  \]

It also means that $L$ is not attained among the indices in $J$. We claim that
$L$ is not attained in $\CR$ at all, from which the last part of the theorem
readily results. Suppose otherwise, so that there exist $k \in \N$ and
$\Omega^*_k \in \CR$ such that $\lambda_k(\Omega_k^*) k^{-1} = L$. By
Lemma \ref{wk}, any connected component of $\Omega^*_k$ is a (contracted
copie of some) minimiser $\Omega^*_m$. Note in particular that $m \in J$. From
Proposition \ref{PropPropagation} follows that $\Omega_k^*$ propagates as a
minimiser, hence $\Omega_m^*$ weakly propagates as a minimiser by definition.
Lemma \ref{propag=weakpropag} implies that $\Omega^*_m$ propagates as a
minimiser, and so $\Omega^*_m$ realises $L$ by Proposition
\ref{PropPropagation}. This is a contradiction.
\end{proof} \vspace{6pt}

The proof of Theorem \ref{polya_stile} for Neumann eigenvalues is a bit more
subtle and this is due to the fact that Neumann eigenvalues do not behave in any
simple way
under inclusion. This is also why Theorem \ref{thm:packing} or modifications of
it fail in that situation : the behaviour under inclusion depends on the eigenfunctions of
the Laplacian. When the quasi-embeddings are actually surjective, however, we
can adapt \cite[Theorem 63]{canz} to our needs.

\begin{lem} \label{Can_est}
Let $V_1, \dots, V_N, W \subset \R^d$ be domains with Lipschitz boundaries.
Assume that $F : V := \sqcup_{j=1}^N \overline{V_j} \to W$ is an isometric
quasi-embedding, which induces a pullback map $F^* : H^1(W) \to H^1(V)$ between
Sobolev spaces. Denote $E_V(k) \subset H^1(V)$ and $E_W(k) \subset H^1(W)$ the
subspaces generated by the first $k$ Neumann eigenfunctions on $V$ and $W$,
respectively. Then for any fixed $k \in \N$, there is a nonzero $\phi \in
E_W(k)$ such that $F^*\phi$ is $L^2$-orthogonal to $E_V(k-1)$ and
\[ \mu_k(V) \, \le \, \frac{\| \phi \|^2_{L^2(W)}}{\| F^*\phi \|^2_{L^2(V)}} \, \mu_k(W) \; .  \]
\end{lem}

\begin{proof}
Let $\{f_k\}_{k \in \N}$ and $\{g_k\}_{k \in \N}$ be $L^2$-orthonormal bases of
Neumann eigenfunctions on $V$ and $W$ respectively, numbered in increasing order
of their eigenvalue. Since $F$ is an isometric quasi-embedding, we can define
pushforwards $F_{\ast}f_k \in L^2(W)$ by extension by $0$ outside the image of
$F$, and $\{F_{\ast}f_k\}_{k \in \N}$ are still $L^2$-orthonormal. Given $k \in
\N$, consider a nonzero linear combination $\phi = \sum_{j=0}^k a_j g_j$, so
that $\| \phi \|^2_{L^2(W)} = \sum_{j=0}^k a_j^2$. The requirement that it be
$L^2$-orthogonal to the first $k$ functions $F_{\ast}f_j$ uniquely specifies
$\phi$ up to a multiplicative constant; we note that $F^{\ast}\phi$ is then
$L^2$-orthogonal to the first $k$ functions $f_j$. On the one hand, we have
\[ \int_W \| \nabla \phi \|^2 dm = \sum_{i,j =0}^k a_i a_j \int_W \langle \nabla
g_i, \nabla g_j \rangle dm = \sum_{j=0}^k a_j^2 \mu_j(W) \le \mu_k(W) \| \phi
\|^2_{L^2(W)} \, ,  \]
while on the other hand
\[ \int_W \| \nabla \phi \|^2 dm = \int_V \| \nabla F^{\ast}\phi \|^2 dm \ge \mu_k(V) \| F^*\phi \|^2_{L^2(V)} \]
due to the fact that $f_0, \dots, f_{k-1}, \phi \in H^1(V)$ generates a
$k$-dimensional subspaces and the variational characterization of $\mu_K(V)$
as the infimum over such subspaces of the maximum of the Rayleigh quotient over
elements of the subspace. The claim readily follows.
\end{proof}

\begin{cor} \label{cor:neumon}
In the context of the previous lemma, if we further assume that $F$ is surjective, then $\mu_k(V) \le \mu_k(W)$ for all $k$.
\end{cor}

\begin{proof}
Since the boundary of $V$ has vanishing Lebesgue measure (being Lipschitz) and since $F$ is an isometry, it follows that $\| \phi \|^2_{L^2(W)} = \| F^*\phi \|^2_{L^2(V)}$.
\end{proof}

We now have all the necessary ingredients to prove Theorem \ref{polya_stile}.
\begin{proof}[Proof of Theorem \ref{polya_stile} for Neumann eigenvalues]
The proof follows the same scheme as the proof of Theorem \ref{polya_stile} for
Dirichlet eigenvalues, using everywhere the corresponding results; notably,
monotonicity is replaced by the Corollary \ref{cor:neumon} and Lemma \ref{wk} is
replaced by Lemma \ref{prf}
\end{proof}

\section{Computational results} \label{sec:comp}

The proposed way of approaching P\'{o}lya's conjecture for a given domain $\Omega$ generates a sequence of 
extremal sets made up of copies of $\Omega$. As we have seen, this sequence encodes information as to whether the
generator set $\Omega$ satisifes the conjecture, which goes beyond whether the corresponding eigenvalues satisfy
inequalities~\eqref{PolyaConj}. These include the behaviour of the number of connected components of the sequence
of extremal sets and the behaviour of the largest scaling coefficient $r_{1,k}$, for instance.

In this section, we present an investigation of the set of ranks for which the
generator is a minimiser for the Dirichlet eigenvalues, and how the above
indicators evolve. 
We chose as generators the disk, the square, and a rectangle of aspect ratio
1:5. The reasons for choosing these generators are as follows.
\begin{itemize}
  \item The exact values of the eigenvalues are known, and can be computed to high accuracy even at high ranks.
    This would not necessarily be the case if we had to approximate eigenvalues
    using, say, finite element methods.
  \item It is not known whether or not the disk satisfies P\'olya's conjecture,
    as opposed to rectangles. This means that we can compare the evolution of
    the indicators in comparison for those two settings.
\end{itemize}

To generate the set of minimisers, we proceed in two steps. The first one
consists in creating a list of eigenvalues for the generators; for the square
and the rectangle this is not a problem since eigenvalues are given by sum of
squares of integers. For the disk the first step consists in generating the zeros of Bessel functions. We denote by
$j_{\nu,k}$ the $k$-th zero of the Bessel function $J_\nu$. The generation of
the list of $j_{\nu,k}$ was
done using the Chebfun MATLAB package \cite{DHT}. Two things were important to
consider:
\begin{itemize}
  \item Bessel functions of high rank $\nu$ are very small (under machine
    precision) but strictly positive
    for a large interval starting at $0$. Root finding algorithms
    would nevertheless find zeros in that range.
  \item All zeros have to be accounted for under a given value.
\end{itemize}
The first point is adressed by using the well-known fact that the first zero of the Bessel
functions $J_\nu$ is always located at some $x > \nu$, and $J_\nu$ is
sufficiently large within that range that no spurious zeros are found. The
second point is adressed by using the property that if $\nu' > \nu$, then for
all $k \in \N$, we
have that $j_{\nu',k} > j_{\nu,k}$. Hence, we can choose as natural stopping
points the first zero of a Bessel function of rank $N$. We then find all
$j_{\nu,k} \in [\nu,j_{N,1}]$ and we can be assured that no zeros have been
skipped. The following pseudo-code will generate the list of Dirichlet
eigenvalues of the disk (keeping in mind that the multiplicity of the
  eigenvalues coming from Bessel zeros of rank $\nu \ge 1$ is $2$, and those
coming from $J_0$ have multiplicity 1). 

\begin{pseudocode}{GenerateDiskEigenvalues}{N}
  bound = \text{First root of $J_N$ above $N$} \\
evalues = \text{All roots $j_{0,k}^2$ between $0$ and $bound$} \\
\FOR \nu = 1 \TO N \DO
evalues = evalues + \text{$2$ copies of all roots $j_{\nu,k}^2$  between  $\nu$ and
$bound$} \\
\RETURN{evalues}
\end{pseudocode}

To find the roots, we used the routine associated with the \emph{chebfun} type
of the aforementionned Chebfun package.

To find the minimisers, we used an approach based on Theorem \ref{wk}. For some
eigenvalue rank, say $k$, the minimiser is either the generator, or, for any
partition of the set of connected components into two subsets, these two subsets
themselves realise $\lambda_j^*(\CR)$ and $\lambda_{j'}^*(\CR)$ for some $j + j'
= k$. Furthermore, in any such case $\lambda_k^* = \lambda_j^* +
\lambda_{j'}^*$. We therefore can find the minimisers recursively, if we have a
list of the eigenvalues of the generator, and a list
of previous minimisers. The following pseudocode will generate such a list under
these conditions; it is defined recursively and outputs a pair consisting of the
list of minimal eigenvalues and a list of the ranks each connected component
making up the minimiser at rank $k$ minimises themselves, according to Theorem
\ref{wk}.

\begin{pseudocode}{$\{\text{minevs,ranks}\}$}{generatorevs,k}
  min = generatorevs[k] \\
  minrank = k \\
  \FOR j = 1 \TO k/2 \DO
  \IF minevs[j] + minevs[k - j] < min \THEN
  \BEGIN
  min = minevs[j] + minevs[k-j] \\
  minrank = j
  \END\\
  minevs[k] = min \\
  \IF minrank == k \THEN ranks[k] = \set{k}
  \ELSE ranks[k] = ranks[minrank] \cup ranks[k - minrank]
\end{pseudocode}

\begin{figure}[ht]
\centering
\includegraphics[width=0.8\textwidth]{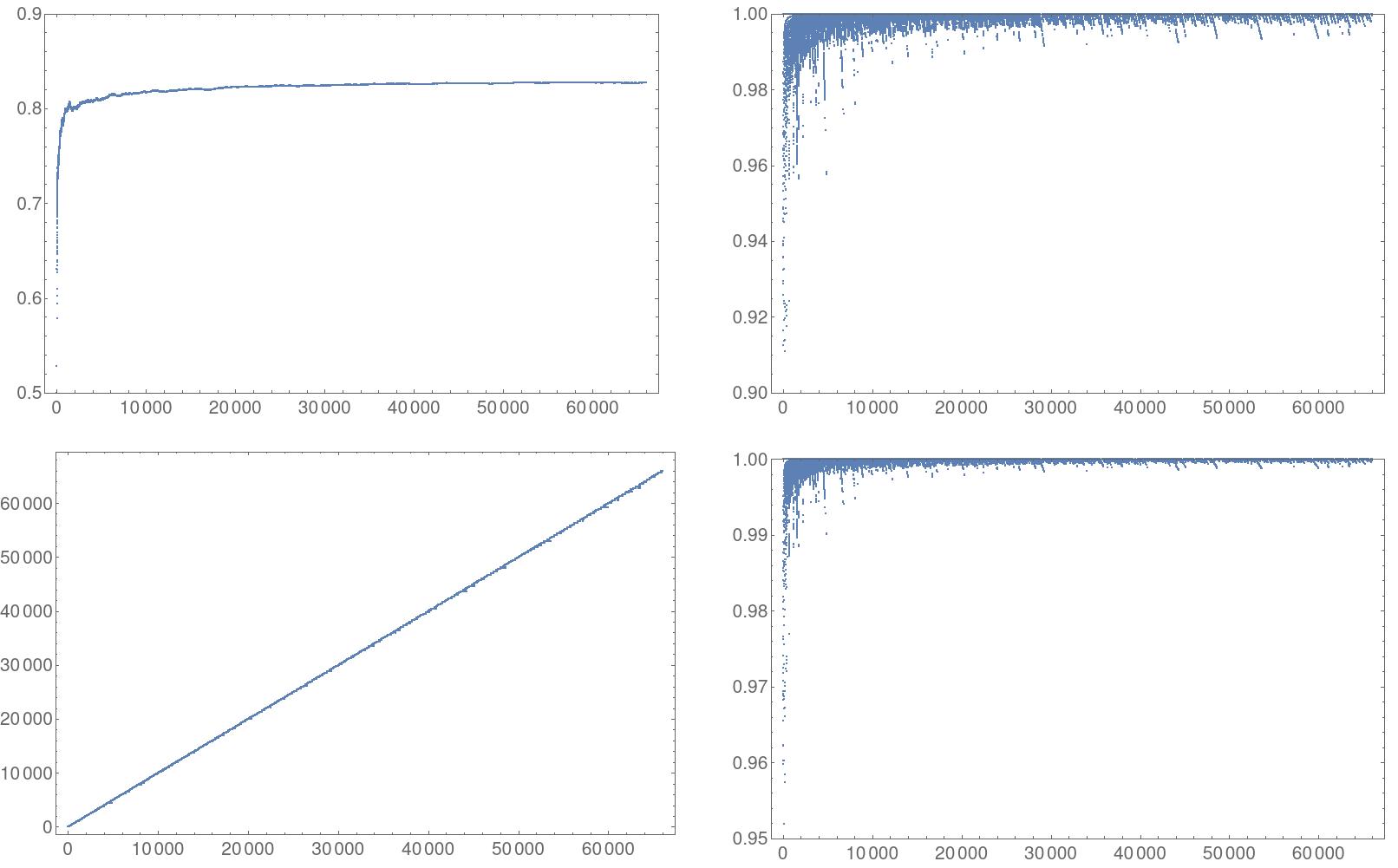}
\caption{Logarithmic density, largest value of coefficient $r_{k}$, largest rank of an eigenvalues on
one connected component and the corresponding logarithmic plot,
in the case of the disk.}
\label{fig:disk}
\end{figure}

The trichotomy in Theorem \ref{thm:trichotomy} indicates that if the generator itself
is a minimiser infinitely often in $\CR$, then P\'olya's conjecture
holds in this case, as well as for any disjoint union of it. As such, we
investigate the log-density of the number ranks for which the generator itself is a minimiser,
that is, the function defined in \eqref{eq:logdensity}. 

Theorem \ref{thm:tfae}
tells us that another indicator to verify is the largest homothety coefficient
$r_k$
of the minimiser, and that the strong P\'olya conjecture is equivalent to this
coefficient converhing to $1$ as $k \to \infty$. As seen in the proof of Theorem
\ref{thm:tfae}, this is implied also by the rank of the maximal eigenvalue
supported by one of the connected component growing asymptotically like $k$.

\begin{figure}[ht]
\centering
\includegraphics[width=0.8\textwidth]{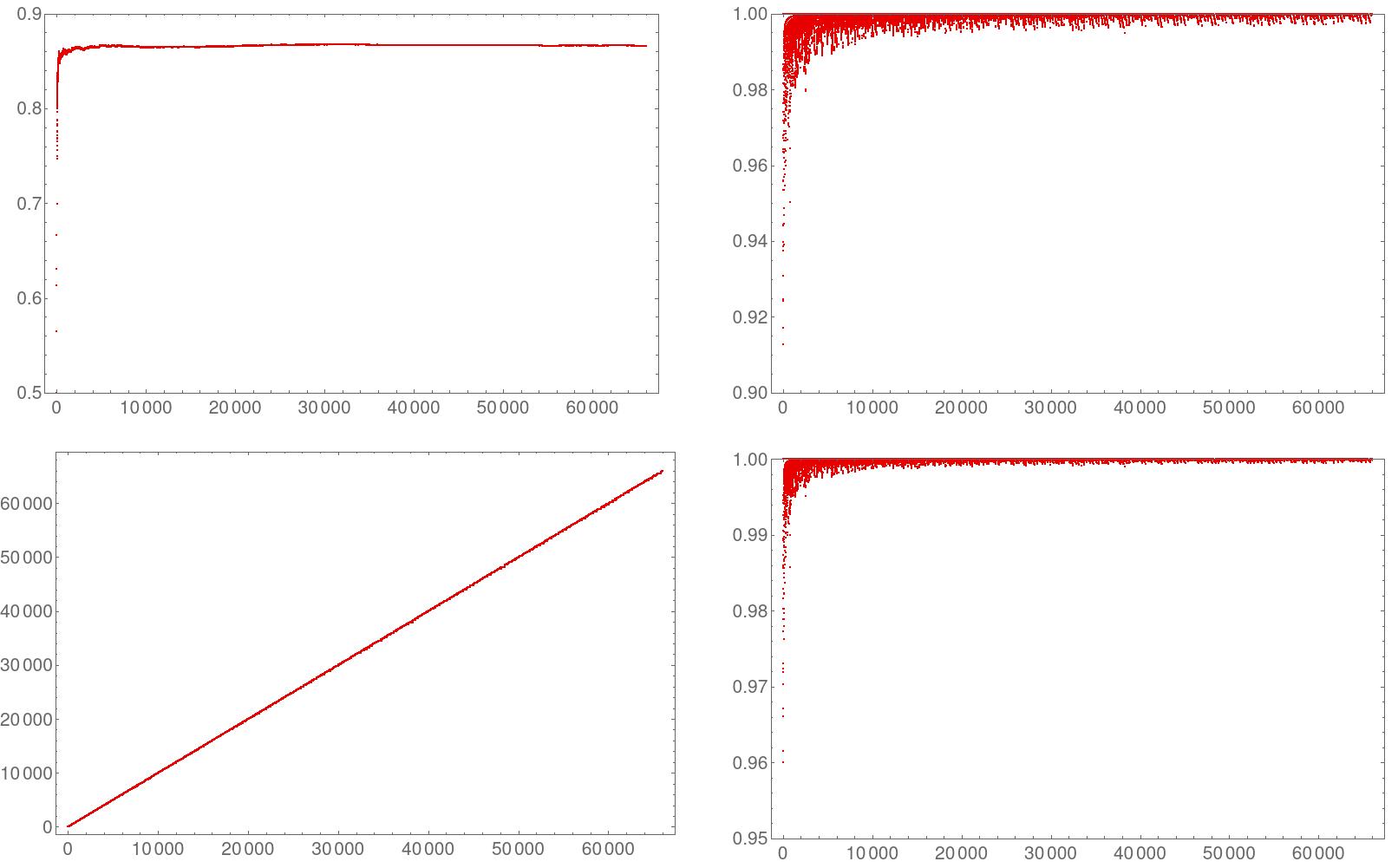}
\caption{Same as in Figure~\ref{fig:square}, now in the case of the square.}
\label{fig:square}
\end{figure}

We show these relevant quantities for the case of the disk in Figure~\ref{fig:disk}, with the
corresponding values for the square being shown in Figure~\ref{fig:square} for comparison. At a first
glance, the qualitative behaviour for these two examples appears to be similar, with the only major difference
that is visible is that the logarithimc density for the disk as a minimiser in the corresponding sequence
appears to be approaching a value somewhat below that of the square.

In view of Theorem \ref{smallo}, another interesting indicator is the number of
connected components of the minimisers. As we have proved, if it grows at
$\smallo k$ rate, $k$ being the eigenvalue rank, then P\'olya's conjecture holds. In the range
of eigenvalues that we investigated,  Figure~\ref{fig:hist} shows that
for the disk, square and a rectangle with side ratio $1:5$, the number of
connected components of the minimisers keeps quite small, both the disk and the
square having a maximum of five components, while the elongated
rectangle exhibits at most only three.

\begin{figure}[ht]
\centering
\includegraphics[width=1\textwidth]{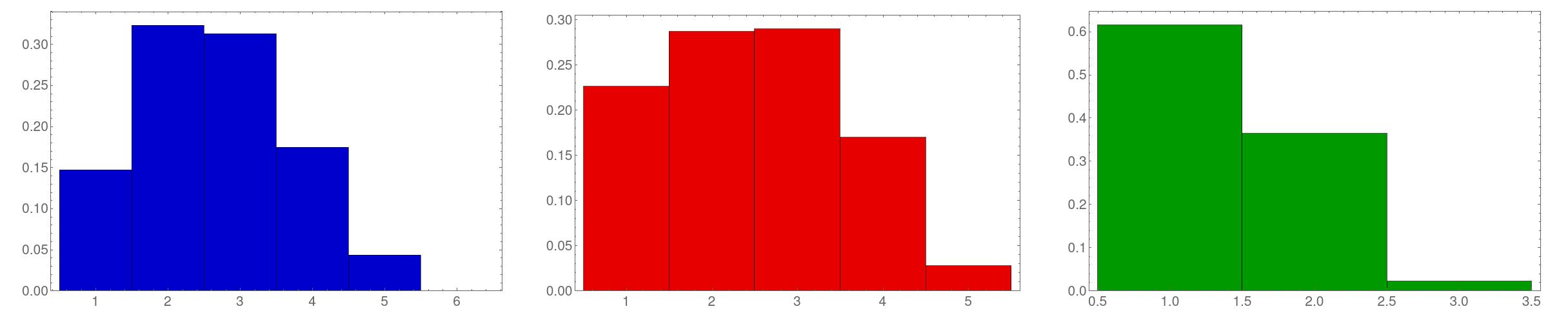}
\caption{Histograms of the number of components: from left to right, disk, square, and
rectangle with sides in the proportion of $1:5$.}
\label{fig:hist}
\end{figure}

Of course, one cannot deduce P\'olya's conjecture from these experiments. However, they show that from the perspective
of the quantities introduced in this paper the behaviour of the disk up to the range considered is not that dissimilar
from that of the square, for instance, which is known to satisfy P\'{o}lya's
conjecture. Furthermore, seeing that the behaviour of these indicators is in
line with P\'olya's conjecture holding, one might hope that it would be easier
to prove indirectly results about the number of connected components of an
extremiser, or about convergence to the generator.



\bibliographystyle{plain}

\end{document}